\documentclass[10pt,twoside,reqno]{amsart}
\usepackage{amsmath,amsfonts,amsthm,amssymb,amscd,epsfig,psfrag}
\usepackage{graphicx}
\usepackage{yhmath}

\usepackage[pdftex, colorlinks=false, backref]{hyperref}
\usepackage{upref}

\theoremstyle{plain}
\newtheorem{theorem}{Theorem}[section]

\newtheorem{lemma}{Lemma}[section]

\theoremstyle{definition}
\newtheorem{definition}{Definition}[section]
\theoremstyle{remark}
\newtheorem{remark}{Remark}[section]

\newcommand{\e}{\varepsilon}
\newcommand{\he}{{\hat{\varepsilon}}}
\newcommand{\s}{\sigma}

\renewcommand{\d}{\partial}

\renewcommand{\u}{{\bf u}}
\newcommand{\ueps}{u^{\e}}
\newcommand{\uheps}{u^{\he}}

\newcommand{\x}{{\bf x}}
\newcommand{\y}{{\bf y}}
\newcommand{\z}{{\bf z}}
\newcommand{\f}{{\bf f}}

\newcommand{\q}{{\bf q}}

\newcommand{\R}{{\mathbb R}}

\newcommand{\norm}[1]{\left\Vert#1\right\Vert}
\newcommand{\abs}[1]{\left|#1\right|}
\newcommand{\Set}[1]{\left\{#1\right\}}

\newcommand{\eps}{\varepsilon}
\newcommand{\px}{\partial_x} 
\newcommand{\py}{\partial_y} 
\newcommand{\sgn}{\operatorname{sgn}}
\newcommand{\hf}{\hat{\f}}

\newcommand{\sh}{\hat{\sigma}}
\newcommand{\hs}{\hat{\sigma}}

\def\be{\begin{equation}}
\def\ee{\end{equation}}
\def\ba{\begin{aligned}}
\def\ea{\end{aligned}}
\def\bes{\begin{equation*}}
\def\ees{\end{equation*}}

\def\bc{\begin{cases}}
\def\ec{\end{cases}}
\numberwithin{equation}{section}

\begin{document}
\title[Stochastic balance laws]{On nonlinear stochastic balance laws}

\author{Gui-Qiang Chen}
\address[Gui-Qiang Chen]{\newline
	Mathematical Institute,
         University of Oxford, 24--29 St Giles',
         Oxford, OX1, 3LB, UK;  Department of Mathematics, Northwestern University,
         Evanston, IL 60208-2730, USA}
\email[]{\href{mailto:chengq@maths.ox.ac.uk}{chengq@maths.ox.ac.uk}}
\urladdr{\href{http://people.maths.ox.ac.uk/chengq/}{http://people.maths.ox.ac.uk/chengq/}}

\author{Qian Ding}
\address[Qian Ding]{\newline
	Department of Mathematics,
         Northwestern University,
         Evanston, IL 60208-2730, USA}
\email[]{\href{mailto:dingqian@math.northwestern.edu}{dingqian@math.northwestern.edu}}

\author{Kenneth H. Karlsen}
\address[Kenneth H. Karlsen]{\newline
   Centre of Mathematics for Applications,
 University of Oslo,
  P.O.\ Box 1053, Blindern,
  NO--0316 Oslo, Norway}
\email[]{\href{mailto:kennethk@math.uio.no}{kennethk@math.uio.no}}
\urladdr{\href{http://folk.uio.no/kennethk}{http://folk.uio.no/kennethk}}

\keywords{Stochastic balance law, vanishing viscosity method, entropy solution,
existence, uniqueness, stability, $BV$ estimates, error estimates, continuous dependence}

\subjclass[2000]{Primary: 35R60, 60H15, 35F25, 35L60; Secondary: 35L03, 60G15}

\date{October 8, 2011}

\begin{abstract}
We are concerned with multidimensional stochastic balance laws.
We identify a class of nonlinear balance laws for which
uniform spatial $BV$ bounds for vanishing viscosity approximations
can be achieved. Moreover, we establish temporal
equicontinuity in $L^1$ of the approximations, uniformly
in the viscosity coefficient. Using these estimates, we
supply a multidimensional existence theory
of stochastic entropy solutions. In addition, we
establish an error estimate for the stochastic viscosity method,
as well as an explicit estimate for the continuous
dependence of stochastic entropy solutions on the flux
and random source functions. Various further generalizations of the results
are discussed.
\end{abstract}
\maketitle

\section{Introduction}\label{sec:intro}

We are concerned with the well-posedness
and continuous dependence estimates for the
stochastic balance laws
\begin{equation}\label{1.1}
	\partial_t u(t,\x)+ \nabla\cdot \f(u(t,\x))
	=\s(u(t,\x))\, \partial_t W(t), \qquad \x\in \R^d,\, t>0,
\end{equation}
with initial data:
\begin{equation}\label{1.2}
	u(0,\x)=u_0(\x),\qquad \x\in \R^d.
\end{equation}
We denote by $\nabla$ and $\Delta$ the spatial gradient
and Laplacian, respectively.

Equation \eqref{1.1} is a conservation law perturbed by a random force driven
by a Brownian motion $W(t)=W(t,\omega)$, $\omega\in \Omega$,  over a stochastic
basis $(\Omega,\mathcal{F},\Set{\mathcal{F}_t}_{t\ge 0},P)$, where $P$ is a
probability measure, $\mathcal{F}$ is a $\sigma$-algebra,
and $\Set{\mathcal{F}_t}_{t\ge 0}$ is a right-continuous filtration
on $(\Omega,\mathcal{F})$ such that $\mathcal{F}_0$ contains all
the $P$--negligible subsets.

The initial function $u_0(\x)$ is assumed to be
a random variable satisfying
\begin{equation}\label{1.3}
	E\left[\|u_0\|_{L^p(\R^d)}^p
	+\abs{u_0}_{BV(\R^d)}\right]<\infty, \qquad p=1,2,\cdots.
\end{equation}

Regarding the flux $\f=(f_1, \cdots, f_d): \R\to \R^d$,
we assume that $f_i\in C^2(\R)$, $i=1,\ldots,d$, and that each $f_i$
has at most polynomial growth in $u$, i.e.,
\begin{equation}\label{eq:f-ass}
	|f_i(u)|\le C \left(1+|u|^r\right)
	\qquad \text{for some finite integer $r\ge 0$.}
\end{equation}

In this paper we focus mainly on the class of noise functions
$\sigma$ for which there exists a constant $C>0$ such that
\begin{equation}\label{eq:sigma-ass}
\sigma(0)=0, \,\,\,\qquad	|\s(u)-\s(v)|\leq C |u-v| \quad \forall u,v\in\R.
\end{equation}
This can be generalized to wider classes for different results in
terms of existence, stability, and continuous dependence, respectively;
see Section \ref{sec:general} for more details. One reason
for requiring $\s(0)=0$ is that it follows from the $L^1$--contraction
principle that  $E[\|u(t,\cdot)\|_{L^1(\R^d)}]$ is finite.
Similarly, the Lipschitz continuity of $\sigma(u)$ is required
for the existence and uniform $L^p$ estimates of solutions.

Stochastic partial differential equations arise in a number
of problems concerning random-phenomena
occurring in biology, physics, engineering, and economics.
In recent years, there has been an increased interest in studying
the effect of stochastic forcing on solutions of nonlinear
stochastic partial differential equations. Of specific interest is
the effect of noise on discontinuous waves, since these are often
the relevant solutions; an issue of particular importance concerns the
well-posedness (existence, uniqueness, and stability) of discontinuous solutions.

The fundamental fluid dynamics models are based on
the compressible Navier-Stokes equations and  Euler equations.
However, abundant experimental observations suggest that the chaotic nature of many
high-velocity fluid dynamics phenomena calls for their stochastic formulation.
Indeed, in these flows with large Reynolds numbers, microscopic
perturbations get amplified to macroscopic scales giving rise to
unsteady flow patterns that deviate significantly from those predicted
by the classical Navier-Stokes/Euler models, and more viable models seem to be
the stochastic Euler or Navier-Stokes equations.
In the present paper we are interested in nonlinear hyperbolic equations
with stochastic forcing, so-called stochastic balance laws.
These balance laws can be viewed as a simple caricature
of the stochastic Euler equations.

Some efforts have been made in the analysis of nonlinear stochastic balance laws.
When $\sigma\equiv 0$, \eqref{1.1} becomes a nonlinear conservation law
for which the maximum principle holds. A satisfactory
well-posedness theory is now available (cf. \cite{Dafermos}).
In \cite{HR}, a one-dimensional stochastic balance law
was analyzed for $u_0$ in $L^\infty$ and compactly
supported $\sigma=\sigma(u)$, which ensures an $L^\infty$
bound. A splitting method was used to construct approximate solutions, and it was
shown that a subsequence of these approximations converges to a
(possible non-unique) weak solution.

For general $\sigma$, the maximum principle is no longer valid.
Indeed, even for  $L^\infty$ initial data $u_0$, the solution is no
longer in $L^\infty$ generically.  For $\sigma=\sigma(t,x)$ in
$C_t(W^{1,\infty}_x)$ and with compact
support in $x$, Kim \cite{Kim} established
the existence and uniqueness of entropy solutions
in the one-dimensional case; see also \cite{Vallet:2009fk}.
For more general $\sigma=\sigma(x,u)$ depending on $u$
and for multidimensional equations in the $L^p$ framework, the uniqueness of
strong stochastic entropy solutions was first established
in Feng-Nualart \cite{FN}, but the existence
result was restricted to one dimension; see the recent
paper Debussche-Vovelle \cite{Debussche:2010fk}
for multidimensional results via a kinetic
formulation\footnote[1]{We became aware of this paper after
our main results were obtained.}. For the $L^p$ theory of
deterministic conservations laws, see \cite{Szepessy:1989vn}.

One of our main observations is that uniform spatial $BV$ bounds are
preserved for stochastic balance laws with noise functions $\sigma(u)$
satisfying \eqref{eq:sigma-ass}. This yields the existence of strong
stochastic entropy solutions in $L^p\cap BV$, as well as in $L^p$,
for multidimensional balance laws \eqref{1.1}. Furthermore, we develop
a ``continuous dependence" theory for stochastic entropy solutions in $BV$, which
can be used, for example, to derive an error estimate for the vanishing viscosity method.
Whenever $\sigma=\sigma(\x,u)$ has a dependency on the spatial position $\x$,
$BV$ estimates are no longer available, but we show that
the continuous dependence framework can be used to derive local
fractional $BV$ estimates, which in turn can be
used, as before via a temporal equicontinuity estimate, to establish
a multidimensional existence result.

Besides providing an existence result in a
multidimensional context by standard methods, one reason for
singling out the class of nonlinear balance laws defined by \eqref{eq:sigma-ass}
is that it makes a natural test bed for numerical analysis, without
having to account for all the added technical complications in a
pure $L^p$ framework. Moreover, by assuming $\s(a)=\s(b)=0$ for
some constants $a<b$, one ensures that the solution remains bounded
between $a$ and $b$ if the initial function $u_0$ does so. Consequently,
it is possible to identify a class of stochastic balance laws for which $L^p\cap BV$,
or even $L^\infty\cap BV$, supplies a relevant and
technically simple functional setting, tailored for the
construction and analysis of numerical methods.

For other related results, we refer to Sinai \cite{Sinai} and
E-Khanin-Mazel-Sinai \cite{EKMS}
for the existence, uniqueness, and
weak convergence of invariant measures for the one-dimensional
Burgers equation with stochastic forcing which is periodic
in $x$, as well as the structure and regularity properties of the
solutions that live on the support of this measure.
We also refer to Lions-Souganidis \cite{LS}
for Hamilton-Jacobi equations with stochastic forcing and the
so-called ``stochastic" viscosity solutions.

We employ the vanishing viscosity method to establish the existence of
stochastic entropy solutions. To this end, consider the
stochastic viscous conservation law
\begin{equation}\label{2.1}
	\partial_t u^\e(t,\x)+ \nabla\cdot \f(u^\e(t,\x))
	=\s(u^\e(t,\x))\partial_t W(t)+\e\Delta u^\e(t,\x)
\end{equation}
for  any fixed $\e>0$, with initial data
\begin{equation}\label{2.2}
	u^\e(0,\x)=u^\e_{0}(\x),\qquad \x\in\R^d,
\end{equation}
where
$u_0^\e(\x)$ is a standard mollifying smooth approximation to $u_0(\x)$
with
$$
E\left[\int_{\R^d} \abs{u_0^\e(\x)}^p\, d\x\right]
\le E\left[\int_{\R^d} |u_0(\x)|^p\, d\x\right]
$$
and, if $u_0\in BV(\R^d)$,
$$
E\left[ \int_{\R^d} \abs{\nabla u_0^\e(\x)} \, d\x\right]
\le E\left[\int_{\R^d} \abs{\nabla u_0(\x)}\, d\x\right].
$$
In addition, $E\left[\int_{\R^d} |\nabla^2 u_0^\e(\x)| \, d\x\right]<\infty$,
i.e., $|\nabla^2 u_0^\e|$ is integrable for each fixed $\e$.

With regard to the viscous equation \eqref{2.1}, we
should replace $(\f,\s)$ by appropriate smooth
approximations $(\f^\eps,\s^\eps)$. However, mainly to ease
the presentation throughout this paper, we will not do that
but instead simply assume that $(\f,\s)$ are sufficiently
smooth (cf. \cite{FN}) in order to ensure the
validity of our calculations.
At times, we will do the same with the initial data.

The existence of global smooth solutions
to \eqref{2.1}--\eqref{2.2} is
established in \cite{FN}, along with the
following uniform estimates for $p\ge 1$ and $T>0$:
\begin{equation}\label{2.3}
	\sup_{\e>0}\sup_{0\leq t\leq T}
	E\left[\|u^\e(t,\cdot)\|_{L^p(\R^d)}^p\right]
	+\sup_{\e>0}E\left[\e\int_0^T
	\|\nabla u^\e(t,\cdot)\|_{L^2(\R^d)}^2dt \right]<\infty.
\end{equation}
The solution satisfies
\begin{equation}\label{2.4}
	\begin{split}
		u^\e(t,\x) & =\int_{\R^d} G_\eps(t, \x-\y)u_0(\y)\, d\y
		\\ & \quad -\int_0^t\int_{\R^d}
		G_\eps(t-s,\x-\y)\nabla\cdot\f(u^\e(t,\y))\, d\y \, ds \\
		&\quad \quad
		+\int_0^t\int_{\R^d} G_\eps(t-s, \x-\y)\sigma(u^\e(s,\y))\, d\y\, dW(s),
	\end{split}
\end{equation}
where $G_\eps(t,\x)$ is the heat kernel:
$$
G_\eps(t,\x)=\frac{1}{(4\pi \e t)^{d/2}}
e^{-\frac{|\x|^2}{4\e t}}, \qquad t>0.
$$
Using \eqref{1.3} and \eqref{2.3}--\eqref{2.4}, it follows that, for each fixed $\e>0$,
\begin{equation}\label{2.5}
	E\left[\|(\nabla, \Delta) u^\e\|_{L^1((0,T)\times \R^d)}\right]
<\infty \qquad \text{for any finite $T>0$,}
\end{equation}
that is, $\nabla u^\e$ and $\nabla^2 u^\e$ are integrable for each fixed $\e>0$.

With different methods, we will later
prove an $\eps$-uniform spatial $BV$ estimate.

The remaining part of this paper is organized as follows:
In Section \ref{sec:BV}, we prove the uniform spatial $BV$
bound for stochastic viscous solutions $\ueps(t,\x)$.
Based on the $BV$ bound, we establish
the equicontinuity of $u^\e(t,\x)$ in $t>0$,
uniformly in the viscosity coefficient $\e>0$, in Section \ref{sec:timecont}.
With these uniform estimates, we establish the existence of stochastic
entropy solutions in $L^p\cap BV$, as the vanishing viscosity limits
for problem \eqref{2.1}--\eqref{2.2} with initial
data in $L^p\cap BV$, in Section \ref{sec:exist-LpBV}.
Combining this existence result with the $L^1$-stability theory in
Feng-Nualart \cite{FN} leads to the well-posedness in $L^p$
for problem  \eqref{1.1}--\eqref{1.2}. We further
establish estimates for the ``continuous dependence on the
nonlinearities" for $BV$ stochastic entropy solutions
in Section \ref{sec:contdep}, which also leads to
an error estimate for  \eqref{2.1}--\eqref{2.2}.
Various further generalizations of the results
are discussed in Section \ref{sec:general}.

\section{Uniform spatial $BV$--estimates}\label{sec:BV}

As indicated in Section \ref{sec:intro}, we have known the regularity
and the uniform $L^p$--estimate \eqref{2.3} ($p\ge 1$)
for the viscous solutions $u^\e(t,\x)$ of \eqref{2.1}--\eqref{2.2}.
In this section, we establish the uniform $L^1$-estimate for $\nabla u^\e$,
that is, the uniform $BV$-estimate of $u^\e(t,\x)$ in the spatial variables $\x$.

Before we do that, let us indicate why $BV$ estimates
do not seem to be available when the noise
coefficient function $\sigma=\sigma(x,u)$ depends on the spatial position $x$,
even if that dependence is $C^\infty$ (see Section \ref{sec:general}
for fractional $BV$ estimates).
To this end, it suffices to consider the simple stochastic differential equation:
$$
du=\sigma(x,u)\, dW(t),
\qquad u(0)=u_0(x), \qquad x\in \R,
$$
where we have dropped nonlinear transport effects
and restricted to one spatial dimension.
The spatial derivative $v=\partial_x u$ satisfies
$$
dv =\left(\sigma_u(x,u)v + \sigma_x(x,u)\right)\, dW(t).
$$
Let $\eta$ be a $C^2$--function. By Ito's formula,
\begin{align*}
	d\eta(v) & =\eta'(v)\big(\sigma_u(x,u)v + \sigma_x(x,u)\big)\, dW(t)
	+\frac12 \eta''(v)\big(\sigma_u(x,u)v
	+ \sigma_x(x,u)\big)^2\, dt.
\end{align*}
Integrating in $x$ and taking expectations, it follows that
\begin{align*}
	E\left[\int \eta(v(t))\, dx\right] & =E\left[\int \eta(v(0))\, dx\right]
	\\ &  \qquad +
	E\left[\int_0^t\int \frac12 \eta''(v)\big(\sigma_u(x,u)v
	+ \sigma_x(x,u)\big)^2\, dx\, ds\right].
\end{align*}
Modulo an approximation argument, we can take $\eta(\cdot)$ as
$\abs{\cdot}$. Unless $\sigma_x\equiv 0$, the second term on the right-hand side
does not seem to be controllable  (this term
vanishes when $\sigma_x\equiv 0$).

\medskip
Let us now continue with the derivation of the $BV$ estimate for \eqref{2.1}.
We will need a $C^2$--approximation of the Kruzkov entropy.
Let $\bar{\eta}:\R\to \R$ be a $C^2$--function satisfying
\begin{equation}\label{2.1b}
	\bar{\eta}(0)=0,\quad \bar{\eta}(-r)=\bar{\eta}(r), \quad
	\bar{\eta}'(-r)=-\bar{\eta}'(r), \quad \bar{\eta}''\ge 0,
\end{equation}
and
\begin{equation}\label{2.2b}
	\bar{\eta}'(r)
	=\begin{cases}
   		-1, & \text{when $r<-1$}, \\
		\in [-1,1], & \text{when $\abs{r}\le 1$}, \\
		+1, & \text{when $r> 1$}.
	\end{cases}
\end{equation}
For any $\rho>0$, define the function $\eta_\delta:\R\to \R$
by
\begin{equation}\label{2.3b}
	\eta_\rho(r)=\rho \bar{\eta}(\frac{r}{\rho}).
\end{equation}
Then
\begin{equation}\label{2.4b}
	\abs{r}-M_1\rho
	\le \eta_\rho(r)\le \abs{r},
	\qquad
	\abs{\eta_\rho''(r)}\le  \frac{M_2}{\rho}
	\mathbf{1}_{\abs{r}<\rho},
\end{equation}
where
\begin{equation}\label{2.5b}
	M_1=\sup_{\abs{r}\le 1}\big|\abs{r}-\bar{\eta}(r)\big|, \qquad
	M_2=\sup_{\abs{r}\le 1}\abs{\bar{\eta}''(r)}.
\end{equation}

We will frequently utilize the
Burkholder-Davis-Gundy inequality, which we now recall.
For $p > 0$, there exists a constant $C=C_p$
such that, if $M_t$ is a continuous
martingale and $t$ a stopping time, then
$$
E\left[ \sup_{s\le t} |M_s|^p \right]
\le C_p E\left[\langle M \rangle_t^{p/2}\right],
$$
where $\langle M \rangle_t$ is the quadratic variation of $M_t$.

\begin{theorem}[Spatial $BV$ estimate]\label{thm:2.1}
Suppose that \eqref{1.3}--\eqref{eq:sigma-ass} hold.
Let $u^\e(t,\x)$ be the solution of \eqref{2.1}--\eqref{2.2}.
Then, for $t>0$,
\begin{equation*}
	E\left[\int_{\R^d} |\nabla u^\e(t,\x)|\,d\x\right]
	\leq E\left[\int_{\R^d}  |\nabla u^\e_0(\x)|\,d\x\right]
	\leq  E\left[\int_{\R^d}  |\nabla u_0(\x)|\,d\x\right].
\end{equation*}
\end{theorem}

\begin{proof}
Taking the derivative of \eqref{2.1} with
respect to $x_i$, $1\leq i\leq d$, we obtain
\begin{equation*}
	\partial_t (u^\e_{x_i})+ \nabla\cdot
	\big(\f'(u^\e(t,\x))u^\e_{x_i}\big)=\s'(u^\e(t,\x))u^\e_{x_i}
	\partial_tW(t)+\e\Delta(u^\e_{x_i}).
\end{equation*}

Applying Ito's formula to $\eta_\rho(u^\e_{x_i})$ yields
\begin{equation}\label{6}
	\begin{split}
		\d_t\eta_\rho(u^\e_{x_i})
		& =\eta_\rho'(u^\e_{x_i})\s'(u^\e)u^\e_{x_i}\d_tW(t)
		\\ & \qquad
		+\eta_\rho'(u^\e_{x_i})\left(\e\Delta u^\e_{x_i}
		-\nabla\cdot(\f'(u^\e)u^\e_{x_i})\right)
		\\ & \qquad\quad
		+\frac{1}{2}\eta_\rho''(u^\e_{x_i})
		\big(\s'(u^\e)u^\e_{x_i}\big)^2.
	\end{split}
\end{equation}

We observe that
\begin{equation}\label{7}
	\begin{split}
		\e \eta_\rho'(u^\e_{x_i})\Delta(u^\e_{x_i})
		& = \e\left(\nabla\cdot (\eta_\rho'(u^\e_{x_i})
		\nabla u^\e_{x_i})-\eta_\rho''(u^\e_{x_i})
		\abs{\nabla u^\e_{x_i}}^2\right)
		\\ & = \e \left(\Delta\eta_\rho(u^\e_{x_i})
		-\eta_\rho''(u^\e_{x_i})|\nabla u^\e_{x_i}|^2\right)\\
		&\le \e \Delta\eta_\rho(u^\e_{x_i}),
	\end{split}
\end{equation}
by using the convexity of $\eta_\rho$, and
interpreting $\Delta\eta_\rho(u^\e_{x_i})$ in the
distributional sense.
Here we have used that $\nabla u_{x_i}^\e, 1\le i\le d$, are integrable (cf.~\eqref{2.5})
so that they vanish at infinity, which leads to the vanishing boundary terms in \eqref{7}.

Integrating \eqref{6} with respect to $\x$, using \eqref{2.5} and \eqref{7}, and noting that
$$
\int_{\R^d}\int_0^t
\eta'(u^\e_{x_i})\s'(u^\e)u^\e_{x_i}\, dW(s)d\x
$$
is a martingale, we arrive at
\begin{equation}\label{8}
	\begin{split}
		&E\left[\int_{\R^d} \eta_\rho(u^\e_{x_i}(t,\x))\, d\x\right]
		-E\left[\int_{\R^d} \eta_\rho(u^\e_{x_i}(0,\x))\,d\x\right]
		\\ & \quad
		\leq E\Biggl[-\int_0^t\int_{\R^d} \eta_\rho'(u^\e_{x_i})
		\nabla\cdot (\f'(u^\e)u^\e_{x_i})\,d\x \, ds
		\\ & \quad \qquad \qquad
		+\frac{1}{2}\int_0^t\int\eta_\rho''(u^\e_{x_i})
		\big(\s'(u^\e)u^\e_{x_i}\big)^2 \,d\x \, ds \Biggr].
	\end{split}
\end{equation}
Now we send $\rho\to 0$ in \eqref{8}.
By the dominated convergence theorem,
\begin{align*}
	& E\left[\int_{\R^d}  \abs{u^\e_{x_i}(t,\x)}\,d\x\right]
	\\ & \leq
	E\left[\int_{\R^d}  \abs{u^\e_{x_i}(0,\x)}\,d\x\right]
	-\lim_{\rho\rightarrow 0}
	E\left[\int_0^t\int_{\R^d}  \eta_\rho'(u^\e_{x_i})\nabla
	\cdot(\f'(u^\e)u^\e_{x_i})\,d\x\, ds\right]
	\\ & \qquad \quad
	+\lim_{\rho\rightarrow 0}\frac{1}{2}
	E\left[\int_0^t\int_{\R^d} \eta_\rho''(u^\e_{x_i})
	\big(\s'(u^\e)u^\e_{x_i}\big)^2\,d\x \, ds\right]
	\\ & := E\left[\int_{\R^d}  \abs{u^\e_{x_i}(0,\x)}\,d\x\right]
	+I_1+I_2.
\end{align*}

For the $I_1$ term,
\begin{align*}
	|I_1| &=\lim_{\rho\to 0}
	\abs{E\left[\int_0^t\int_{\R^d}  \nabla\cdot\big(\f'(u^\e)
	\eta_\rho'(u^\e_{x_i})u^\e_{x_i}\big)\,d\x\, ds\right]}
	\\ &\quad \quad +\lim_{\rho\to 0}
	\abs{E\left [\int_0^t\int_{\R^d}  \eta_\rho''(u^\e_{x_i})u_{x_i}^\e
	\nabla u_{x_i}^\e\cdot\f'(u^\e)\, d\x\, ds\right]}\\
	& \le C \lim_{\rho\to 0}
	E\left[\int_0^t\int_{\R^d}  \abs{u_{x_i}^\e}
	\, \frac{1}{\rho}\chi_{[-\rho, \rho]}(u_{x_i}^\e)
	\abs{\nabla u_{x_i}^\e} \abs{\f'(u^\e)} \,d\x \, ds\right].
\end{align*}
Notice that
$$
\abs{u_{x_i}^\e}\, \frac{1}{\rho} \chi_{[-\rho, \rho]}(u_{x_i}^\e) \to 0
\qquad \text{for a.e.~$(t,\x)$ almost surely  as $\rho\to 0$},
$$
and
\begin{align*}
	& \abs{u_{x_i}^\e}
	\frac{1}{\rho}\chi_{[-\rho, \rho]}(u_{x_i}^\e)
	\abs{\nabla u_{x_i}^\e} \abs{\f'(u^\e)}
	\\ & \qquad
	\le C \left(\abs{\nabla u_{x_i}^\e}^2
	+  \abs{u^\e}^{2(r-1)}\right),
\end{align*}
where the right-side term of the inequality is integrable and independent of $\rho>0$.
Then the dominated convergence theorem implies that
$|I_1|=0$.

Next we consider $I_2$.
By condition \eqref{eq:sigma-ass}
and estimate \eqref{2.4b}, we have
\begin{align*}
	\abs{\eta_\rho''(u^\e_{x_i})(\s'(u^\e)u^\e_{x_i})^2}
	&=\abs{\eta_\rho''(u^\e_{x_i})}\abs{|u^\e_{x_i}}^2(\s'(u^\e))^2
	\\ & \leq C \abs{u^\e_{x_i}} {\mathbf 1}_{\{|u^\e_{x_i}|<\rho\}}
	\leq C \abs{u^\e_{x_i}} \in L^1((0, T)\times\R^d).
\end{align*}
On the other hand, since $|u_{x_i}^\e|$ is integrable and independent of $\rho>0$ and
$$
\abs{u^\e_{x_i}} {\mathbf 1}_{\abs{|u^\e_{x_i}}<\rho\}} \to 0
\qquad \text{for a.e.~$(t,\x)$ almost surely as $\rho\to 0$},
$$
the dominated convergence
theorem again implies $|I_2|=0$.
This concludes the proof.
\end{proof}

\section{Uniform Temporal $L^1$--Continuity}\label{sec:timecont}

In this section, we establish the uniform temporal $L^1$--continuity
of $u^\e(t,\x)$, independent of the
viscosity coefficient $\e>0$.

\begin{theorem}[Temporal $L^1$--continuity]\label{thm:time-cont}
Suppose that \eqref{1.3}--\eqref{eq:sigma-ass} hold.
Let $u^\e(t,\x)$ be the solution of \eqref{2.1}--\eqref{2.2}.
Let $D\subset \R^d$ be a bounded domain in $\R^d$ and $T>0$ finite.
Then, for any small $\Delta t>0$, there exists
a constant $C>0$ independent of $\Delta t$ such that

\begin{eqnarray}
&&E\left[\int_0^{T-\Delta t}\int_D
	\abs{u^\e(t+\Delta t,\x)-u^\e(t,\x)}\, d\x dt\right]\nonumber\\[1.5mm]
&&\quad \leq C(\Delta t)^{1/3}\to 0
\qquad \text{as $\Delta t\to 0$}. \label{eq:time-cont}
\end{eqnarray}
\end{theorem}

\begin{proof}
Fix $\Delta t>0$. For $t\in[0,T-\Delta t]$, set
$w^\e(t,\cdot):=u^\e(t+\Delta t,\cdot)-u^\e(t,\cdot)$.
Then, for any $\varphi\in L^\infty(0, T; C_0^\infty(D))$, we have
\begin{equation} \label{difference}
	\begin{split}
		& \int_D w^\e(t, \x)\varphi(t,\x)\, d\x\\
		&=\int_D \Big(\int_{t}^{t+\Delta t}\partial_s u^\e(s,\x)ds\Big)\varphi(t,\x)\,d\x
		\\ &= \int_t^{t+\Delta t}
		\int_D \f(u^\e(s,\x))\cdot\nabla\varphi(t,\x)\, d\x\, ds
		\\ & \qquad -\e\int_t^{t+\Delta t}\int_D
		\nabla u^\e(s,\x)\cdot \nabla\varphi(t, \x)\, d\x\, ds\\
		&\qquad\qquad +\int_t^{t+\Delta t}\int_D
		\s(u^\e(s,\x))\varphi(t, \x)\,d\x\, dW(s).
	\end{split}
\end{equation}

For each $t\in[0,T-\Delta t]$, take $\delta>0$, set
$D_{-\delta}:=\{\x\in D\, :\, \text{dist}(\x,\d D)\geq\delta\}$,
and denote by $\chi_{D_{-\delta}}(\cdot)$ its
characteristic function.

Let $J\in C_c^\infty(\R^d)$ be the standard mollifier defined by
\begin{equation}\label{3.1}
	\begin{split}
		J(\x)=
		\begin{cases}
			C\exp\left(\frac{1}{|\x|^2-1}\right)
			\quad & \text{if $|\x|<1$},\\
			0 & \text{if $|\x|\geq 1$},
		\end{cases}
	\end{split}
\end{equation}
where the constant $C>0$ is chosen so that
$\int_{\R^d}J(\x)d\x=1$.
For each $\delta>0$, we take
$$
\varphi:=\varphi_\delta(t,\x)
=\delta^{-d}\int_{\R^d} J(\tfrac{\x-\y}{\delta})
\sgn\left(w(t, \y)\right)\chi_{D_{-\delta}}(\y)\, d\y
$$
in \eqref{difference}. It is clear that
$\norm{\varphi_\delta}_{L^\infty(D)}
+\delta\norm{\nabla\varphi_\delta}_{L^\infty(D)}\leq C$,
uniformly in $t$, for some constant $C>0$
independent of $\delta>0$.

Integrating \eqref{difference} in $t$  from $0$ to $T-\Delta t$ yields
\begin{equation*}
	\begin{split}
		&\int_0^{T-\Delta t}\int_D \abs{w^\e(t,\x)}\, d\x dt
		\\ & =\int_0^{T-\Delta t}\int_t^{t+\Delta t}
		\int_D \f(u^\e(s,\x))\cdot\nabla\varphi_\delta(t,\x)\, d\x\, ds\, dt
		\\ & \quad -\int_0^{T-\Delta t}\int_t^{t+\Delta t}
		\int_D\e \nabla u^\e(s,\x)\cdot \nabla\varphi_\delta(t, \x)\, d\x\, ds\, dt
		\\ & \quad +\int_0^{T-\Delta t}\left(\int_t^{t+\Delta t}
		\Big(\, \int_D \s(u^\e(s,\x))\varphi_\delta(t, \x) \, d\x \Big) dW(s)\right) dt
		\\ & \quad+\int_0^{T-\Delta t}\int_D w^\e(t, \x)\big(w^\e(t,\x)
		-\varphi_\delta(t,\x)\big) \, d\x dt\\
		&:= \sum_{j=1}^4 I_j^\delta.
	\end{split}		
\end{equation*}

We examine these parts separately.

Thanks to the polynomial growth of $\f$ and \eqref{2.3},
\begin{equation*}
	\abs{E\left [I_1^\delta\right ]}
	\leq C\frac{\Delta t}{\delta}\norm{\f}_{L^1(D\times (0,T))}
	\leq C(T,D)\frac{\Delta t}{\delta}.
\end{equation*}

For the term $I_{2}^\delta$,  we have
\begin{equation*}
	\begin{split}
		\abs{E\left [I_2^\delta\right ]}
		&\le C\left( E\left[\int_0^{T-\Delta t}
		\int_D \left(\int_t^{t+\Delta t}\sqrt{\e}|\nabla u^\e(s,\x)|\, ds\right)^2
		\, d\x\, dt\right]\right)^{\frac{1}{2}}
		\\ & \qquad \qquad \times
		\left(E\Big[\int_0^{T-\Delta t}\int_D \e |\nabla \varphi_\delta |^2\, d\x\, ds\Big]
		\right)^{\frac{1}{2}} \\
		&\leq  C \Delta t\left( E\Big[\int_0^{T-\Delta t}
		\int_D |\nabla \varphi_\delta|^2 d\x\, ds\Big]\right)^{\frac{1}{2}}\\
		&\leq  C(T,D)\frac{\Delta t}{\delta},
	\end{split}
\end{equation*}
where the second inequality follows from the energy estimate \eqref{2.3}:
$$
\sup_{\e>0} E\left[\e\int_0^T
\norm{\nabla u^\e(t,\x)}_{L^2(\R^d)}^2\, dt\right]<\infty.
$$

For the term $I_3^\delta$, by the Burkholder-Davis-Gundy inequality applied
to the martingale $0\le \Delta t\mapsto\int_t^{t+\Delta t}
\left (\, \int_D \s(u^\e(s,\x))\varphi_\delta(t, \x) \, d\x \right) dW(s)$,
we have
\begin{equation*}
	\begin{split}
		\abs{E\left [I_3^\delta\right ]}
		& \leq C \int_0^{T-\Delta t}
		E\left[\left(\int_t^{t+\Delta t} \left (\, \int_D \s(u^\e(s,\x)) \varphi_\delta(t,\x)
		\, d\x \,\right  )^2\, ds\right)^{\frac{1}{2}}\right]\, dt
		\\ & \leq  C\left(E\left[\int_0^{T-\Delta t}\int_t^{t+\Delta t}
		\int_D \big(\s(u^\e(s,\x) \varphi_\delta(t,\x)\big)^2
		\, d\x\, ds\, dt\right ]\right)^{\frac{1}{2}}
		\\ & \leq  C\left(E\left[\int_0^{\Delta t}\int_0^{T-\Delta t}
		\int_D \big(\s(u^\e(s+t,\x)\big)^2
		\, d\x\, dt\, ds\right ]\right)^{\frac{1}{2}}
		\\ & \leq C\sqrt{\Delta t}
		\left( E\left[\int_0^{T}
		\int_D \big(\s(u^\e(t,\x)\big)^2 d\x\, dt \right]\right)^{\frac{1}{2}}
		\\ & \leq C\sqrt{\Delta t}
		\left( E\left[\int_0^{T}
		\int_D |u^\e(t,\x)|^2 d\x\, dt \right]\right)^{\frac{1}{2}}\\
		& \leq C\sqrt{\Delta t},
	\end{split}
\end{equation*}
where we have used that $\sup_{\e>0} E\left[\|u^\e(t)\|_2^2\right]<\infty$,
uniformly in $t> 0$.

This $L^2$--bound also implies
\begin{align*}
	& E\left[\int_0^T\int_{D\setminus D_{-2\delta}} |u^\e(t,\x)|\, d\x \, dt\right]
	\\ & \quad
	\le C\, \left(E\left[\|u^\e\|_2^2\right]\right)^{\frac{1}{2}}
	\left(E\big[\int_0^T\int_{D\setminus D_{-2\delta}}\, d\x\, dt\big]
	\right)^{\frac{1}{2}}\\
&\quad \le C \sqrt{\delta}.
\end{align*}
Hence,
\begin{align*}
	&\abs{ E\left[I_4^\delta\right]}\\
	&\le  2E\left [\int_0^{T-\Delta t}
	\int_{D\setminus D_{-2\delta}} |w(t,\x)|\, d\x\, dt\right ]\\
	&\quad
	+ E\Biggl[\int_0^{T-\Delta t}
	\int_{D_{-2\delta}}\Bigl | \, |w(t,\x)|
	-w(t,\x) \int_{\R^d} \delta^{-d}J(\tfrac{\x-\y}{\delta})
	\sgn\big(w(t,\y)\big)\Bigr | \, d\y\, d\x\, dt\Biggr]
	\\ & \leq C\sqrt{\delta}\\
	&+E\Biggl[\int_0^{T-\Delta t}
	\int_{D_{-2\delta}}\int_{\R^d}
	\delta^{-d}J(\tfrac{\x-\y}{\delta})
	\Bigl | \,|w(t,\x)|-w(t,\x)\sgn(w(t,\y))\Bigr|
	\, d\y\, d\x\, dt \Biggr]
	\\ & \le C\sqrt{\delta}
	\\ & \quad
	+CE\left[\int_0^{T-\Delta t}\int_{D_{-2\delta}}
	\int_{\R^d} \delta^{-d}J(\tfrac{\x-\y}{\delta})
	|w(t,\x)- w(t,\y)|\, d\y\, d\x\, dt\right]
	\\ & \le
	C\sqrt{\delta}
	+CE\left[\int J(z)\int_0^{T}\int_{D_{-2\delta}}
	|u^\e(t,\x)-u^\e(t,\x-\delta z)|\, d\x\, dt \, dz\right]
	\\ & \le  C\delta^{\frac{1}{2}}+4\delta
	\leq  C\sqrt{\delta},
\end{align*}
where the third inequality follows from
$\big||a|-a \sgn(b)\big|\le 2|a-b|$ for any $a,b\in\R$.
The fifth inequality follows,
since $u^\e$ belongs to $BV$ in $\x$.

Setting $\rho(\Delta t)=\inf_{\delta>0}\left\{C_1\frac{\Delta t}{\delta}
+C_2(\Delta t)^{\frac{1}{2}}
+C_3\delta^{\frac{1}{2}}\right\}$, it follows that
$$
\int_0^{T-\Delta t}\int_D
|w(t,\x)| \,d\x dt \le \rho (\Delta t).
$$
The function $\rho(\cdot)$ reaches the
infimum at $\delta=C(\Delta t)^{\frac{2}{3}}$,
and hence
$$
\int_0^{T-\Delta t}
\int_D |w(t,\x)| \,d\x \,dt
\le C(\Delta t)^{\frac{1}{3}}\to 0
\quad \text{as $\Delta t\to 0$.}
$$
\end{proof}

\begin{remark}
Since Brownian sample paths
are $\alpha$-H\"older continuous
for every $\alpha<\frac{1}{2}$, a fractional
order in the temporal $L^1$--continuity
in \eqref {eq:time-cont} is expected. The proof
of Theorem \ref{thm:time-cont} uses an
idea due to Kruzkov \cite{Kruzkov}.
\end{remark}

\medskip
\section{Well-Posedness Theory in  $L^p$}\label{sec:exist-LpBV}

Before we introduce the relevant notions of
generalized solutions, let us define what is
meant by an entropy-entropy flux pair $(\eta, \q)$, or
more simply an entropy pair, namely a $C^2$
function $\eta:\R\to\R$ such that $\eta', \eta''$ have at
most polynomial growth, with corresponding
entropy flux $\q$ defined by $\q'(u)=\eta'(u) \f'(u)$.
An entropy pair is called convex if $\eta''(u)\ge 0$.

\begin{definition}[Stochastic entropy solutions]\label{def:ES}
A $\Set{\mathcal{F}_t}_{t\ge 0}$--adapted, $L^2(\R^d)$--valued stochastic
process $u=u(t,\x;\omega)$ is a \textit{stochastic entropy solution}
of the balance law \eqref{1.1} with initial data \eqref{1.2} provided
that the following conditions hold:

\begin{enumerate}
	\item [(i)] for $p=1,2,\cdots$,
	$$
	\sup_{0\le t\le T} E\left[\norm{u(t)}_{L^p(\R^d)}^p\right]<\infty,
	\quad \text{for any $T>0$};
	$$

	\item [(ii)] for any convex entropy pair $(\eta, \q)$ and any $0<s<t$,
	\begin{equation*}
		\begin{split}
			& -\left(\, \int_{\R^d} \eta(u(t,\x))\varphi(\x) \, d\x
			- \int_{\R^d} \eta(u(s,\x))\varphi(\x) \, d\x\,  \right)
			\\ & \qquad
			+ \int_s^t \int_{\R^d} \q(u(\tau,\x))\cdot\nabla \varphi \,d\x\,d\tau
			\\ & \qquad \quad
			+ \int_s^t \int_{\R^d} \frac12 \eta''(u(\tau,\x))
			\big(\sigma(u(\tau,\x))\big)^2\varphi \,d\x\,d\tau
			\\ & \qquad\quad\quad
			+ \int_s^t \left(\, \int_{\R^d} \eta'(u(\tau,\x))
			\sigma(u(\tau,\x)) \varphi \,d\x\, \right) dW(\tau) \ge 0,
		\end{split}
	\end{equation*}
	for all $\varphi\in C^\infty_c(\R^d)$, $\varphi\ge 0$, where
	$\int_s^t \left(\cdots\right)\, dW(\tau)$ is an Ito integral.
\end{enumerate}
\end{definition}

To motivate the next definition, let us make a formal attempt to derive the
$L^1$--contraction property for stochastic entropy solutions. To this end,
consider smooth (in $x$) solutions to the one-dimensional problems:
$$
du + \px f(u)\, dt = \sigma(u) \, dW, \qquad u|_{t=0}=u_0,
$$
$$
dv + \px f(v)\, dt = \sigma(v) \, dW, \qquad v|_{t=0}=v_0.
$$
Subtracting the two stochastic conservation laws yields
$$
d(u-v) = - \left[\px (f(u)-f(v))\right]\, dt+ \left[\sigma(u)-\sigma(v)\right] \, dW.
$$
Let $\eta(\cdot)$ be an entropy. An application
of the chain rule (Ito's formula) now yields
\begin{align*}
	d\eta(u-v) & =
	\Biggl[ -\px \bigl(\eta'(u-v) (f(u)-f(v))\bigr)
	\\ & \qquad \qquad
	+ \eta''(u-v)\big(f(u)-f(v)\big)\px (u-v)
	\\ & \qquad\qquad\qquad
	+\frac12\eta''(u-v) \big(\sigma(u)-\sigma(v)\big)^2  \Biggr] \, dt
	\\ & \qquad \qquad\qquad \qquad
	+ \eta'(u-v)\big(\sigma(u)-\sigma(v)\big)\, dW,
\end{align*}
where the last term is a martingale. Choosing $\eta(\cdot)=\abs{\cdot}$ yields
$\eta''(\cdot)=\delta_0$ and the two ``$\eta''$ terms" vanish.
Consequently, after integrating and taking expectations, we arrive at the
$L^1$--contraction (conservation) principle:
$$
E\left[\int |u(t)-v(t)|\,dx\right]
 =E\left[\int |u_0-v_0|\,dx\right].
$$

Of course, for non-smooth solutions, the Ito formula is not available and
we should instead derive the $L^1$--contraction principle
from the (stochastic) entropy inequalities via Kruzkov's method.

Attempting precisely that, we
write the entropy condition for $u(t)=u(t,x;\omega)$ with the entropy
$\eta(u(t)-v(s,y;\omega))$, where $v(s,y;\omega)$ is being treated as a constant
with respect to $(t,x)$. Similarly, write the entropy condition
for $v(s)=v(s,y;\omega)$ for the entropy $\eta(v(s)-u(t,x;\omega))$, with
$u(t,x;\omega)$ being constant with respect to $(s,y)$.
Take $\eta(\cdot)=\abs{\cdot}$, and
then $q(u,v)=\sgn(u-v)(f(u)-f(v))$. After adding
together the two entropy inequalities, we formally obtain
\begin{align*}
	&(d_t+d_s)|u-v|
	\\ & \quad
	\le \Big[- \left(\px+\py\right) \bigl(\sgn(u-v)(f(u)-f(v))\bigr)
	\\ & \qquad \qquad
	+\frac12\delta(u-v)\left[\left(\sigma(u)\right)^2
	+\left(\sigma(v)\right)^2\right]  \Big]\, dt\, ds
	\\ & \qquad\qquad\qquad
	+ \sgn(u(t,x)-v(s,y)) \sigma(u(t,x))\, dW(t)\, ds
	\\ & \qquad\qquad\qquad
	- \sgn(u(t,x)-v(s,y)) \sigma(v(s,y))\, dW(s)\, dt.
\end{align*}

Depending on $t<s$ or $t>s$, one of the last two terms are not
adapted, and this causes a problem for the Ito integral. In particular, by taking
the expectation of the above inequality, only one of the last two terms vanishes.
Moreover, to write
$\frac12\delta(u-v)\left[\left(\sigma(u)\right)^2+\left(\sigma(v)\right)^2\right]$
in the favorable form:
$$
\frac12\delta(u-v)\left(\sigma(u)-\sigma(v)\right)^2,
$$
we are missing the cross term $2\sigma(u)\sigma(v)$.
These difficulties can be effectively handled by
the notion of ``strong" stochastic entropy solutions.

\begin{definition}[Strong stochastic entropy solutions]\label{def:strong-ES}
An $\Set{\mathcal{F}_t}_{t\ge 0}$--adapted, $L^2(\R^d)$--valued stochastic
process $u=u(t)=u(t,\x;\omega)$ is a \textit{strong stochastic entropy solution}
of the balance law \eqref{1.1} with initial data \eqref{1.2} provided
$u$ is a stochastic entropy solution, and
the following additional condition holds:
\begin{enumerate}
	\item[(iii)] for each $\Set{\mathcal{F}_t}_{t\ge 0}$--adapted, $L^2(\R)$--valued
	stochastic process $\tilde{u}=\tilde{u}(t)=\tilde{u}(t,\x;\omega)$ satisfying
	$$
	\sup_{0\le t\le T} E\left[\norm{\tilde{u}(t)}_{L^p(\R^d)}^p\right]<\infty
	\qquad
	\text{for any $T>0$, $p=1,2,\cdots$,}
	$$
	and for each entropy function $S:\R\to\R$, with
	$$
	\overline{S}(r;v,\y)
	:=\int_{\R^d} S'(\tilde u(r,\x)-v)
	\sigma(\tilde u(r,\x))\varphi(\x,\y)\, d\x,
	$$
	where $r\ge 0$, $v\in\R$, $\y\in \R^d$, and
	$\varphi \in C^\infty_c(\R^d\times\R^d)$, there exists a
	deterministic function $\Delta(s,t)$, $0\le s\le t$, such that
	\begin{align*}
		\qquad\qquad
		& E\left[\int_{\R^d} \int_s^t \overline{S}(\tau; v=u(t,\y),\y)
		\, dW(\tau)\, d\y  \right]
		\\ & \le E\left[ \int_s^t \int_{\R^d}
		\partial_v \overline{S}(\tau; v=\tilde{u}(\tau,\y),\y)
		\sigma(u(\tau,\y))\, d\y \, d\tau \right] + \Delta(s,t),
	\end{align*}
	where $\Delta(\cdot,\cdot)$ is such that, for each $T>0$, there exists
	a partition $\Set{t_i}_{i=1}^m$ of $[0,T]$, $0=t_0<t_1<\dots < t_m=T$,
	so that
	$$
	\lim_{\max\limits_i \abs{t_{i+1}-t_i}}
	\sum_{i=1}^m \Delta(t_i,t_{i+1})=0.
	$$
\end{enumerate}
\end{definition}

The notion of strong stochastic entropy solutions is due to
Feng-Nualart \cite{FN}, who proved the
$L^1$--contraction property for these solutions:
\begin{equation}\label{eq:FN-contr}
	E\left[\|u(t)-v(t)\|_{L^1(\R^d)}\right]
	\le E\left[\|u_0-v_0\|_{L^1(\R^d)}\right]
	\qquad \text{for $t>0$,}
\end{equation}
where $u(t)$ is any stochastic entropy solution with $u|_{t=0}=u_0$ and $v(t)$ is
any \textit{strong} stochastic entropy solution with $v|_{t=0}=v_0$.
In \eqref{eq:FN-contr}, the entropy $\abs{\cdot}$ can be replaced by
$(\cdot)^+$, yielding the $L^1$--comparison principle.

Feng-Nualart \cite{FN} employed the compensated compactness
method to prove an existence result in the one-dimensional context.
The following theorem provides the existence of
strong stochastic entropy solutions for a
class of multidimensional equations.

\begin{theorem}[Existence in $L^p\cap BV$]\label{thm:existence-BV}
Suppose that \eqref{1.3}--\eqref{eq:sigma-ass} hold.
Then there exists a strong stochastic
entropy solution $u$ of the balance law \eqref{1.1}
with initial data \eqref{1.2}, satisfying
\begin{equation}\label{eq:Lp-BV}
	E\left[|u(t,\cdot)|_{BV(\R^d)}\right]
	\le E\left[\abs{u_0}_{BV(\R^d)}\right]
	\qquad \text{for any $t\ge 0$.}
\end{equation}
\end{theorem}

\begin{proof}
For fixed $\e>0$, we mollify $u_0$ by $u_0^\e\in C^\infty$ so that
$E\left[\|u_0^\e\|_{H^s(\R^d)}^2\right]$ is finite for any $s>0$, and
$$
E\left[\|u_0^\e\|_{L^p(\R^d)}^p +|u^\e_0|_{BV(\R^d)}\right]
\le E\left[\|u_0\|_{L^p(\R^d)}^p + |u_0|_{BV(\R^d)}\right]<\infty,
$$
for any $p=1,2, \cdots$, and $u_0^\e(\x)\to u_0(\x)$
for a.e.~$\x$, almost surely as $\e\to 0$.

Now the same arguments as in Section 4
of Feng-Nualart \cite{FN} yield that
there exists an $\mathcal{F}_t$--adapted
stochastic process $\ueps=u^\e(t)\in
C([0,\infty); L^2(\R^d))$
satisfying almost surely that
\begin{enumerate}
\item[(i)] $E\left[\|u^\e(t,\cdot)\|_{H^s(\R^d)}^2\right]<\infty$ for all $t>0$;

\smallskip
\item[(ii)] $\partial_{x_ix_j}u^\e(t, \cdot)\in C(\R^d)$ for all $i,j=1, \ldots, d$;

\smallskip
\item[(iii)] For any $\varphi\in C_c^\infty(\R^d), \varphi\ge 0$,
and $0<s<t$,
\begin{align*}
	& \langle\eta(u^\e(t,\cdot)), \varphi\rangle-\langle\eta(u^\e(s,\cdot)), \varphi\rangle
	\\ & =
	\int_s^t\left \langle\q(u^\e(\tau,\cdot)), \nabla\varphi \right\rangle \, d\tau
	+\frac{1}{2}\int_s^t\left \langle\eta''(u^\e(\tau, \cdot))
	(\sigma(u^\e(\tau,\cdot))^2, \varphi \right \rangle \, d\tau
	\\ & \qquad +\int_s^t\left \langle\eta'(u(\tau,\cdot))\sigma(u(\tau, \cdot)), \varphi
	\right \rangle \, dW(\tau)
	\\ & \qquad \qquad+\e\int_s^t\Bigl(\left\langle \eta(u^\e(\tau,\cdot)),
	\Delta\varphi\right\rangle
	-\left\langle\eta''(u^\e(\tau,\cdot))|\nabla u^\e(\tau,\cdot)|^2,
	\varphi\right \rangle\Bigr)\, d\tau
	\\ & \le \int_s^t\left\langle\q(u^\e(\tau,\cdot)), \nabla\varphi\right\rangle\, d\tau
	+\frac{1}{2}\int_s^t\left \langle \eta''(u^\e(\tau, \cdot))
	(\sigma(u^\e(\tau,\cdot))^2, \varphi\right \rangle\, d\tau
	\\ & \qquad +\int_s^t\left \langle\eta'(u^\e(\tau,\cdot))
	\sigma(u(\tau, \cdot)), \varphi\right \rangle\, dW(\tau)
	+\mathcal{O}(\e),
\end{align*}
\end{enumerate}
where the first equality in (iii) follows from the Ito formula.

Combining the results established in
Sections \ref{sec:BV} and \ref{sec:timecont}, we conclude that
there exist a subsequence (still denoted) $\{u^\e(t,\x)\}_{\e>0}$ and
a limit $u(t,\x)$ such that as $\e\to 0$,
\begin{equation*}
	\text{$u^\e(t,\x)\to u(t,\x)\qquad$ for a.e.~$(t,\x)$, almost surely},
\end{equation*}
and the limit $u(t,\x)$ satisfies \eqref{eq:Lp-BV}. Arguing as
in Feng-Nualart \cite{FN}, we can pass to the limit
in the entropy inequality (iii) to conclude that
the limit function $u(t,\x)$ is a stochastic entropy
solution (cf.~Definition \ref{def:ES}).
Moreover, we can prove that $u$ is
a \textit{strong} stochastic entropy
solution, as defined in Definition \ref{def:strong-ES}.
\end{proof}

Combining Theorem 4.1 with the $L^1$--stability
result established in Feng-Nualart \cite{FN}, we conclude

\begin{theorem}[Well-posedness in $L^p$]
Suppose \eqref{eq:f-ass} and \eqref{eq:sigma-ass} hold, and
that $u_0$ satisfies
\begin{equation*}
	E\left[\|u_0\|_{L^p(\R^d)}^p\right]<\infty,
	\quad p=1,2, \cdots.
\end{equation*}

\begin{enumerate}
\item[(i)] Existence: There exists a strong stochastic entropy
solution of the balance law \eqref{1.1}
with initial data \eqref{1.2}, satisfying for any $t\ge 0$,
\begin{equation}\label{5.2}
	E\left[\|u(t,\cdot)\|_{L^p(\R^d)}^p\right]
	<\infty, \qquad p=1,2, \cdots.
\end{equation}

\item[(ii)] Stability:  Let $u(t,\x)$ be a strong stochastic entropy
solution of \eqref{1.1} with initial data $u_0(\x)$, and
let $v(t,\x)$ be a stochastic entropy solution
with initial data $v_0(\x)$. Then, for any $t>0$,
\begin{equation}\label{eq:L1-stab}
	E\left[\int_{\R^d} \abs{u(t,\x)-v(t,\x)}\,d\x\right]
	\le E\left[\int_{\R^d} \abs{u_0(\x)-v_0(\x)}\,d\x\right].
\end{equation}
\end{enumerate}
\end{theorem}

\begin{proof}
For the $\cap_{p=1}^\infty L^p(\R^d)$-valued random
variable $u_0$, we can approximate $u_0$ by
$u_0^\delta(\x)$ in $L^1$ as $\delta\to 0$,
with $E[\|u_0^\delta\|_{p}^p+ |u_0^\delta|_{BV}]<\infty$ for fixed $\delta>0$.
Then Theorem \ref{thm:existence-BV} indicates that
there exists a corresponding family of global strong entropy solutions
$u^\delta(t,\x)$ for $\delta>0$.

Then the $L^1$--stability (contraction) result established
in Feng-Nualart \cite{FN} implies that $u^\delta(t,\x)$ is a Cauchy
sequence in $L^1$, which yields the strong convergence
of $u^\delta(t,\x)$ to $u(t,\x)$ a.e., almost surely.
Since
$$
E\left[\|u^\delta(t,\cdot)\|_{L^p(\R^d)}^p\right]
\le E\left[\|u^\delta_0(\cdot)\|_{L^p(\R^d)}^p\right]\le C,
\quad p=1,2, \cdots,
$$
where $C$ is independent of $\delta$, one can check that
$u(t,\x)$ is a strong stochastic entropy solution, and \eqref{5.2} holds.
For the stability result \eqref{eq:L1-stab}, see \cite{FN}.
\end{proof}

\section{Continuous dependence estimates}\label{sec:contdep}

The aim of this section is to establish an
explicit ``continuous dependence on the nonlinearities"
estimate in the $BV$ class. Let $u(t)=u(t,\x;\omega)$ be a
strong stochastic entropy solution of
\begin{equation}\label{eq:u}
	\partial_t u + \nabla\cdot \f(u) =
	\sigma(u) \, \partial_tW, \qquad u|_{t=0}=u_0.
\end{equation}
Let $v(t)=v(t,\x;\omega)$ be a strong stochastic entropy solution of
\begin{equation}\label{eq:v}
	\partial_t v + \nabla\cdot \hf(v) =
	\sh(v) \,  \partial_t W, \qquad v|_{t=0}=v_0.
\end{equation}
We are interested in estimating $E\left[\norm{u(t)-v(t)}_{L^1}\right]$ in terms
of $u_0-v_0$, $\f-\hf$, and $\s-\hs$. Relevant
continuous dependence results for deterministic conservation laws
have been obtained in \cite{Lucier:1986px,Bouchut:1998ys}, and
in \cite{Cockburn:1999fk} for strongly degenerate parabolic equations;
see also \cite{Chen:2005wf,Karlsen:2003za}.

We start with the following important lemma.

\begin{lemma}\label{lem:contdep}
Suppose that \eqref{1.3}--\eqref{eq:sigma-ass} hold
for the two data sets $(u_0,\f,\s)$ and $(v_0,\hf,\hs)$.
For any fixed $\e>0$, let $u(t)=u(t,\x;\omega)$ be the
solution to the stochastic parabolic problem
\begin{equation}\label{eq:ue}
	d u+ \bigl[ \nabla_\x\cdot \f(u)-\e\Delta_\x u\bigr]\, dt
	=\s(u)\, dW(t),
	 \qquad u|_{t=0}=u_0.
\end{equation}
For any fixed $\he>0$, let $v(t)=v(t,\y;\omega)$ be
the solution to the stochastic parabolic problem
\begin{equation}\label{eq:ve}
	dv+ \bigl[\nabla_\y\cdot \hf(v) - \he\Delta_\y v\bigr]\, dt
	=\sh(v)\, dW(t),
	 \quad v|_{t=0}=v_0.
\end{equation}

Take $0\le \phi_\delta=\phi_\delta(\x,\y)
\in C^\infty_c(\R^d\times \R^d)$ to be of the form:
\begin{equation}\label{eq:test-function}
	\phi_\delta(\x,\y)
	=\tfrac{1}{\delta^{d}} J(\tfrac{\x-\y}{2\delta})\psi(\tfrac{\x+\y}{2})
	=: J_\delta(\tfrac{\x-\y}{2})\psi(\tfrac{\x+\y}{2}),
\end{equation}
where $J(\cdot)$ is a regularization kernel as in \eqref{3.1}
and $0\le \psi\in C^\infty_c(\R^d)$. Moreover,
given any entropy function $\eta(\cdot)$ with
$\eta(0)=0$ and $\eta'(\cdot)$ odd, introduce the
associated entropy fluxes for $u,v\in \R$:
$$
\q^\f(u,v)=\int_v^u \eta'(\xi-v)\f'(\xi)\, d\xi, \quad
\q^{\hf}(u,v)=\int_v^u \eta'(\xi-v)\hf'(\xi)\, d\xi.
$$

Then, for any $t>0$,
\begin{align*}
	& \iint \eta(u(t,\x)-v(t,\x))\phi_\delta(\x,\y) d\x\,d\y
	-\iint \eta(u_0(\x)-v_0(\y))\phi_\delta(\x,\y)\,d\x\,d\y
	\\ &
	\le I^{\f}(\phi_\delta)+I^{\f,\hf}(\phi_\delta)
	+I^{\s,\hs}(\phi_\delta)+I^{\e,\he}(\phi_\delta)
	\\ & \,\,\,\,\,
	+ \iint \!\! \int_s^t \eta'(u(s,\x)-v(s,\y))
	\bigl( \s(u(s,\x))-\hs(v(s,\y))\bigr)\phi_\delta(\x,\y)\, dW(s) \,d\x\,d\y,
\end{align*}
where
\begin{align*}
	& I^{\f}(\phi_\delta) = \iint  \!\!\int_0^t  \q^\f(u(s,\x),v(s,\y))\cdot \nabla
	\psi(\tfrac{\x+\y}{2}) J_\delta(\tfrac{\x-\y}{2})\,ds \,d\x\,d\y,
	\\  & I^{\f,\hf}(\phi_\delta)  = \iint\!\!  \int_0^t
	\bigl(\q^{\hf}(v(s,\y),u(s,\x))
	-\q^\f(u(s,\x),v(s,\y))\bigr)\cdot
	\nabla_\y \phi_\delta(\x,\y)\, ds\,d\x\,d\y,
	\\ & I^{\e,\he}(\phi_\delta)
	= \big(\sqrt{\e}-\sqrt{\he}\big)^2\!\!
	\iint \!\! \int_0^t
	\eta(u(s,\x)-v(s,\y))  \Delta_\y J_\delta(\tfrac{\x-\y}{2})
	\psi(\tfrac{\x+\y}{2}) \, ds\,d\x\,d\y
	\\ &\qquad \quad\quad\,\,
	+ \frac14 \big(\sqrt{\e}+\sqrt{\he}\big)^2\!\!
	\iint \!\! \int_0^t \eta(u(s,\x)-v(s,\y))
	J_\delta(\tfrac{\x-\y}{2})
	\Delta\psi(\tfrac{\x+\y}{2})  \, ds\,d\x\,d\y
	\\ & \qquad\quad\quad\,\,
	+\big(\he-\e \big) \!\!  \iint \!\! \int_0^t \eta(u(s,\x)-v(s,\y))
	\nabla_\y J_\delta(\x-\y) \cdot \nabla \psi(\tfrac{\x+\y}{2})
	 \, ds\,d\x\,d\y,
	\\  & I^{\s,\hs}(\phi_\delta)
	=\iint \!\! \int_0^t \frac12 \eta''(u(s,\x)-v(s,\y))
	\\ & \qquad\qquad\qquad\qquad\qquad
	\times \bigl( \s(u(s,\x))-\hs(v(s,\y))\bigr)^2
	\phi_\delta(\x,\y)\,ds\,d\x\,d\y.
\end{align*}
\end{lemma}

\begin{proof}
Subtracting \eqref{eq:ve} from \eqref{eq:ue} and
subsequently applying Ito's formula
to $\eta\bigl(u(t)-v(t)\bigr)$, we obtain
\begin{equation}\label{eq:eta(u-v)-I}
	\begin{split}
		d\eta(u-v) & =
		\Big[\! \!-\eta'(u-v)\bigl( \nabla_\x \cdot \f(u)-\nabla_\y \cdot\hf(v) \bigr)
		+ \eta'(u-v)\bigl( \e\Delta_\x u - \he \Delta_\y v\bigr)
		\\ & \qquad\qquad
		+\frac12\eta''(u-v) \big(\sigma(u)-\sigma(v)\big)^2  \Big] \, dt
		\\ & \qquad\qquad\qquad
		+ \eta'(u-v)\big(\sigma(u)-\sigma(v)\big)\, dW(t).
	\end{split}
\end{equation}

Observe that
$$
\eta'(u-v)\nabla_\x\cdot \f(u)=\nabla_\x \cdot \q^\f(u,v),\quad
\eta'(u-v)\nabla_\y\cdot \hf(v)=\nabla_\y \cdot \q^{\hf}(v,u),
$$
and thus
\begin{align*}
	&-\eta'(u-v)\bigl( \nabla_\x \cdot \f(u)-\nabla_\y \cdot\hf(v) \bigr)
	\\ & \qquad \quad
	=-(\nabla_\x+\nabla_\y )\cdot \q^\f(u,v)
	+\nabla_\y \cdot \bigl(\q^\f(u,v)-\q^{\hf}(v,u)\bigr).
\end{align*}
Next,
\begin{align*}
	&\eta'(u-v)\bigl( \e\Delta_\x u - \he \Delta_\y v\bigr)\\
	& = \bigl(\e \Delta_\x+\he \Delta_y\bigr) \eta(u-v) -
	\eta''(u-v)\bigl(\e |\nabla_\x u|^2 + \he |\nabla_\y v|^2\bigr)
	\\ & = \bigl(\e \Delta_\x+2\sqrt{\e} \sqrt{\he} \nabla_\x \cdot \nabla_\y
	+ \he \Delta_y\bigr) \eta(u-v)
	-  \eta''(u-v)\bigl|\sqrt{\e} \nabla_\x u - \sqrt{\he} \nabla_\y v\bigr|^2.
\end{align*}

Inserting the last two relations into \eqref{eq:eta(u-v)-I}, we arrive at
\begin{equation}\label{eq:eta(u-v)-II}
	\begin{split}
		d\eta(u-v) & =
		\Big[ -(\nabla_\x+\nabla_\y)\cdot \q^\f(u,v)
		+\nabla_\y \cdot \bigl(\q(u,v)-\q^{\hf}(v,u)\bigr)
		\\ & \quad\quad\quad
		+\bigl(\e \Delta_\x+2\sqrt{\e} \sqrt{\he} \nabla_\x \cdot \nabla_\y
		+ \he \Delta_\y\bigr) \eta(u-v)
		\\ & \quad\quad\quad\quad
		-\eta''(u-v)\bigl|\sqrt{\e} \nabla_\x u - \sqrt{\he} \nabla_\y v\bigr|^2
		\\ & \quad \quad\quad\quad\quad
		+\frac12\eta''(u-v) \big(\sigma(u)-\sigma(v)\big)^2  \Big] \, dt
		\\ & \quad\quad\quad\quad\quad\quad
		+ \eta'(u-v)\big(\sigma(u)-\sigma(v)\big)\, dW(t).
	\end{split}
\end{equation}

We integrate \eqref{eq:eta(u-v)-II} against the test
function $\phi_\delta$ defined
in \eqref{eq:test-function}, yielding
\begin{align*}
	&\iint \eta(u(t,\x)-v(t,\x))\phi_\delta(\x,\y) d\x\,d\y
	-\iint \eta(u_0(\x)-v_0(\y))\phi_\delta(\x,\y)\,d\x\,d\y\qquad
	\\ &\le
	I_c^1 + I_c^2 + I_d + I^{\s,\hs}(\phi_\delta)
	\\ &\quad
	+ \iint \!\! \int_s^t \eta'(u(s,\x)-v(s,\y))
\bigl( \sigma(u(s,\x))-\sigma(v(s,\y))\bigr)
	\phi_\delta(\x,\y)\, dW(s) \,d\x\,d\y,
\end{align*}
where
\begin{align*}
	I_c^1 & := -\iint\!\!  \int_0^t (\nabla_\x+\nabla_\y)\cdot \q^\f(u,v)
	\phi_\delta(\x,\y) \, ds\,d\x\,d\y,
	\\
	I_c^2 & := \iint\!\!  \int_0^t \nabla_\y \cdot
	\bigl(\q^\f(u(s,\x),v(s,\y))-\q^{\hf}(v(s,\y),u(s,\x))\bigr)
	\phi_\delta(\x,\y)\, ds\,d\x\,d\y,
	\\
	I_d & :=  \iint \!\! \int_0^t
	\bigl(\e \Delta_\x+2\sqrt{\e} \sqrt{\he} \nabla_\x \cdot \nabla_\y
	+ \he \Delta_\y\bigr) \eta(u(s,\x)-v(s,\y))\phi_\delta(\x,\y)\, ds\,d\x\,d\y.
\end{align*}

Integrating by parts gives $I_c^2= I^{\f,\hf}(\phi_\delta)$, and also
$I_c^1=I^{\f}(\phi_\delta)$, since
$$
(\nabla_\x +\nabla_\y) \phi=J_\delta(\tfrac{\x-\y}{2})
 (\nabla_\x +\nabla_\y) \psi(\tfrac{\x+\y}{2})
 = J_\delta(\tfrac{\x-\y}{2})
 \nabla \psi(\tfrac{\x+\y}{2}).
 $$

We now investigate the term $I_d$. A calculation shows that
\begin{align*}
	& \bigl(\e \Delta_\x+2\sqrt{\e} \sqrt{\he} \nabla_\x \cdot \nabla_\y
	+ \he \Delta_\y\bigr)\phi(\x,\y)
	\\ & = \bigl(\e \Delta_\x+2\sqrt{\e} \sqrt{\he} \nabla_\x \cdot \nabla_\y
	+ \he \Delta_\y\bigr)J_\delta(\tfrac{\x-\y}{2})\psi(\tfrac{\x+\y}{2})
	\\ & \qquad
	+J_\delta(\x-\y)\bigl(\e \Delta_\x+2\sqrt{\e} \sqrt{\he} \nabla_\x \cdot \nabla_\y
	+ \he \Delta_\y\bigr)\psi(\tfrac{\x+\y}{2}) + R,
\end{align*}
and
\begin{align*}
	&R= 2 \e \nabla_\x J_\delta(\tfrac{\x-\y}{2}) \cdot \nabla_\x  \psi(\tfrac{\x+\y}{2})
	+ 2 \he \nabla_\y J_\delta(\x-\y) \cdot \nabla_\y  \psi(\tfrac{\x+\y}{2})
	\\ & \qquad
	+ 2\sqrt{\e}\sqrt{\he} \nabla_\x J_\delta(\tfrac{\x-\y}{2}) \cdot \nabla_\y\psi(\tfrac{\x+\y}{2})
	+ 2 \sqrt{\e}\sqrt{\he} \nabla_\y J_\delta(\tfrac{\x-\y}{2}) \cdot \nabla_\x  \psi(\tfrac{\x+\y}{2})
	\\ &  = \Bigl(2 \e \nabla_\x J_\delta(\tfrac{\x-\y}{2})
	+2 \sqrt{\e}\sqrt{\he} \nabla_\x J_\delta(\tfrac{\x-\y}{2})
	+ 2 \sqrt{\e}\sqrt{\he} \nabla_\y J_\delta(\tfrac{\x-\y}{2})
	\\ & \qquad\qquad
	+ 2 \he \nabla_\y J_\delta(\tfrac{\x-\y}{2}) \Bigr)\cdot \nabla_\y  \psi(\tfrac{\x+\y}{2})
	\\ & = 2\nabla_\y J_\delta(\tfrac{\x-\y}{2}) \cdot \nabla_\y \psi(\tfrac{\x+\y}{2})\big(\he-\e \big)
	= \nabla_\y J_\delta(\tfrac{\x-\y}{2}) \cdot\nabla \psi(\tfrac{\x+\y}{2}) \big(\he-\e \big).
\end{align*}
Moreover,
\begin{align*}
	&\bigl(\e \Delta_\x+2\sqrt{\e} \sqrt{\he} \nabla_\x \cdot \nabla_\y
	+ \he \Delta_\y\bigr)J_\delta(\x-\y)
	=  \big(\sqrt{\e}-\sqrt{\he}\big)^2 \Delta_\y J_\delta(\x-\y),
	\\ &
	\bigl(\e \Delta_\x+2\sqrt{\e} \sqrt{\he} \nabla_\x \cdot \nabla_\y
	+ \he \Delta_\y\bigr)\psi(\tfrac{\x+\y}{2})
	=  \frac14 \big(\sqrt{\e}+\sqrt{\he}\big)^2
	\Delta\psi(\tfrac{\x+\y}{2}).
\end{align*}
Consequently, after integrating by parts, $I_d$ becomes $ I^{\e,\he}(\phi_\delta)$.
\end{proof}

\begin{theorem}[Continuous dependence estimates]\label{conti-dependence}
Suppose that \eqref{1.3}--\eqref{eq:sigma-ass} hold
for the two data sets $(u_0,\f,\s)$ and $(v_0,\hf,\hs)$.
Let $u(t)$ and $v(t)$ be the strong stochastic entropy
solutions of \eqref{eq:u}--\eqref{eq:v}, respectively, for which
$$
E\left[ \abs{v(t)}_{BV(\R^d)}\right]
\le E\left[ \abs{v_0}_{BV(\R^d)} \right]
\qquad \text{for $t>0$.}
$$
In addition, we assume that either
$$
u,v\in L^\infty((0,T)\times\R^d\times\Omega) \qquad \text{for any $T>0$,}
$$
or
$$
 \f'', \f'-\hf', \s-\sh \in L^\infty.
$$

Then
\begin{itemize}
\item[(i)] there is a constant $C_T>0$ such that,
for any $0<t<T$ with $T$ finite,
\begin{align*}
	& E\left[\int_{\R^d} |u(t,\x)-v(t,\x)|\psi(\x)\, d\x\right] \\
	& \quad \le C_{T}\Biggl( E\left[\int_{\R^d} |u_0(\x)-v_0(\x)|\psi(\x) \,d\x\right]
	+ \sqrt{t} \|\psi\|_{L^1(\R^d)}
	\|\sigma-\sh\|_{L^\infty}\\
	 & \qquad\qquad\qquad\quad
	 + t\, E\left[ \abs{v_0}_{BV(\R^d)} \right]
	 \left(\|\f'-\hf'\|_{L^\infty} +\|\sigma-\sh\|_{L^\infty}\right)\Biggr),
\end{align*}
where the constant $C_T>0$ is independent of $\abs{u_0}_{BV(\R^d)}$ and
$\abs{v_0}_{BV(\R^d)}$, and may grow exponentially in $T$.
Moreover, $\psi=\psi(\x)\ge 0$ is any function
satisfying $\abs{\psi} \le C_0$, $\abs{\nabla\psi}\le C_0\psi$,
which includes $\psi(\x)=e^{-C_0|\x|}$ and, more generally,
$\psi(\x)=1$ when $|\x|\le R$ and
$\psi(\x)=e^{-C_0(|\x|-R)}$ when $|\x|\ge R$.
In particular,  for any $R>0$, this choice implies
\begin{align*}
	& E\left[\int_{\abs{\x}<R} |u(t,\x)-v(t,\x)|\, d\x\right] \\
	& \quad \le C_{T,R}
	\Biggl( E\left[\int_{\R^d} |u_0(\x)-v_0(\x)| \,d\x\right]
	+ \sqrt{t} \|\sigma-\sh\|_{L^\infty}\
	\\ & \qquad\qquad\qquad\quad
	+ t\, E\left[ \abs{v_0}_{BV(\R^d)} \right]
	\left(\|\f'-\hf'\|_{L^\infty} +\|\sigma-\sh\|_{L^\infty}\right)\Biggr).
\end{align*}

\item[(ii)] There is a constant $C_T$ such that,
for any $0<t<T<\infty$,
\begin{align*}
	& E\left[\int_{\R^d} |u(t,\x)-v(t,\x)|\psi(\x) \, d\x\right] \\
	& \quad \le C_{T}\Biggl( E\left[\int_{\R^d} |u_0(\x)-v_0(\x)|\psi(\x) \,d\x\right]
	+ \sqrt{t} \|\psi\|_{L^1(\R^d)} \Delta(\s,\hs)\\
	& \qquad\qquad\qquad\quad
	+  t\, E\left[ \abs{v_0}_{BV(\R^d)} \right]
	\left(\|\f'-\hf'\|_{L^\infty}+\Delta(\s,\hs)\right)\Biggr),
\end{align*}
where $\psi(\x)$ is as before and
\begin{equation*}
	\Delta(\sigma,\sh):=
	\sup_{\xi\neq 0} \frac{\abs{\sigma(\xi)-\sh(\xi)}}{\abs{\xi}}.
\end{equation*}
\end{itemize}
\end{theorem}

\begin{remark}
If, in addition to the assumptions listed
in Theorem \ref{conti-dependence}, $u_0(\x)$
and $v_0(\x)$ are periodic in $\x$ with the same period,
we can  ``remove" $\psi$ from the above
estimates, since integrations are then over a bounded domain.
\end{remark}

\begin{proof}
As the vanishing viscosity method
converges (cf.~Theorem \ref{thm:existence-BV}),
it suffices to prove the result for \eqref{eq:ue}--\eqref{eq:ve} with $\he=\e$.

For $\rho>0$, let $\eta_\rho:\R\to \R$ be the function defined by
\eqref{2.1b}--\eqref{2.5b}. Then the function
$$
\q_\rho^\f(u,v)=\int_v^u \eta_\rho '(\xi-v)\f'(\xi)\, d\xi,
\qquad u,v\in \R,
$$
satisfies
\begin{equation}\label{eq:entropy-prop}
	\abs{\partial_u \left(\q_\rho^\f(u,v)-\q_\rho^\f(v,u)\right)}
	\le\frac{M_2}{2} \norm{\f''}_{L^\infty}\rho,
\end{equation}
where $M_2=\sup_{\abs{u}\le 1}\abs{\bar{\eta}''(u)}$.

In view of Lemma \ref{lem:contdep} with $\he=\e$,
\begin{equation}\label{eq:uv-entropy1}
	\begin{split}
		&E\left[\iint \eta_\rho(u(t,\x)-v(t,\y))\phi_\delta(\x,\y)\,d\x\, d\y\right]
		\\ & \qquad
		-E\left[\iint \eta_\rho(u_0(\x)-v_0(\x))\phi_\delta(\x,\y)\,d\x\,d\y\right]
		\\& \le E\left[\iint  \!\!\int_0^t  \q^\f_\rho(u(s,\x),v(s,\y))\cdot
		\nabla \psi(\tfrac{\x+\y}{2}) J_\delta(\tfrac{\x-\y}{2})\,ds \,d\x\,d\y\right]
		\\ & \quad
		+ E\Biggl[\iint  \!\! \int_0^t \bigl(\q^{\hf}_\rho(v(s,\y),u(s,\x))
		\\ & \qquad \qquad \qquad\qquad
		-\q^\f_\rho(u(s,\x),v(s,\y))\bigr)\cdot \nabla_\y\phi_\delta\, ds\,d\x\,d\y\Biggr]
		\\ & \quad
		+ E\Biggl[\iint\!\!\int_0^t \frac12 \eta_\rho''(u(s,\x)-v(s,\y))
		\\ & \qquad \qquad \qquad\qquad
		\times \bigl(\s(u(s,\x))-\hs(v(s,\y))\bigr)^2\phi_\delta(\x,\y)\, ds\,d\x\,d\y\Biggr]
		\\ & \quad 	+ \, \e\, E\Biggl[\iint \!\! \int_0^t
		\eta_\rho(u(s,\x)-v(s,\y))  J_\delta(\tfrac{\x-\y}{2})
		\Delta_\x\psi(\tfrac{\x+\y}{2})  \, ds\,d\x\,d\y\Biggr].
	\end{split}
\end{equation}

Observe that
\begin{align*}
	& -\nabla_\y\cdot \bigl(\q^{\hf}_\rho(v(s,\y),u(s,\x))
	-\q^\f_\rho(u(s,\x),v(s,\y))\bigr)
	\\ & \qquad
	=\nabla_\y v \cdot \partial_v \bigl(\q^\f_\rho(u,v)
	-\q^{\hf}_\rho(v,u)\bigr)\big |_{(u,v)=(u(s,\x),v(s,\y))},
\end{align*}
and, thanks to \eqref{eq:entropy-prop},
\begin{align*}
	&\abs{ \partial_v \bigl(\q^\f_\rho(u,v)-\q^{\hf}_\rho(v,u)\bigr)}
	\\ & \qquad
	= \abs{\partial_v \bigl(\q^\f_\rho(v,u)-\q^{\hf}_\rho(v,u) \bigr)
	+ \partial_v \bigl(\q^\f_\rho(u,v)-\q^{\f}_\rho(v,u) \bigr)}
	\\ & \qquad
	\le | \f'(v)-\hf'(v)|+ \frac{M_2}{2} \|\f''\|_{L^\infty}\rho.
\end{align*}
Hence, after an integration by parts,
\begin{align*}
	& \abs{E\left[\iint  \!\! \int_0^t \bigl(\q^{\hf}_\rho(v(s,\y),u(s,\x))
	-\q^\f_\rho(u(s,\x),v(s,\y))\bigr)\cdot \nabla_\y\phi_\delta
	\, ds\,d\x\,d\y\right]}
	\\ & \qquad
	\le t\,  E\left[ \abs{v_0}_{BV(\R^d)} \right]
	\norm{\psi}_{L^\infty(\R^d)}
	\Big(\|\f'-\hf'\|_{L^\infty}+ \frac{M_2}{2}\|\f''\|_{L^\infty}\rho \Big).
\end{align*}

Consequently, again thanks to \eqref{eq:entropy-prop}
and also \eqref{2.4b},
we can write \eqref{eq:uv-entropy1} as
\begin{equation}\label{eq:uv-entropy2}
	\begin{split}
		&E\left[\iint |u(t,\x)-v(t,\y)| \phi_\delta(\x,\y)\,d\x\, d\y\right]
		\\ & \qquad
		-E\left[\iint |u_0(\x)-v_0(\x)| \phi_\delta(\x,\y)\,d\x\,d\y\right]
		\\& \le E\left[\iint  \!\!\int_0^t  \q_\rho^\f(u(s,\x),v(s,\y))\cdot
		\nabla \psi(\tfrac{\x+\y}{2})
		J_\delta(\tfrac{\x-\y}{2})\,ds \,d\x\,d\y\right]
		\\ & \quad
		+ E\Biggl[\iint\!\!\int_0^t \frac12 \eta_\rho''(u(s,\x)-v(s,\y))
		\\ & \qquad\qquad\qquad\qquad
		\times \bigl(\s(u(s,\x))-\hs(v(s,\y))\bigr)^2\phi_\delta(\x,\y)\, ds\,d\x\,d\y\Biggr]
		\\ & \qquad\qquad
		 + t\,  \abs{v_0}_{BV(\R^d)} \norm{\psi}_{L^\infty(\R^d)}
		 \left( \|\f'-\hf'\|_{L^\infty} +\mathcal{O}(\rho)\right)
		\\ & \qquad\qquad\qquad
		+ \mathcal{O}\left(\norm{\psi}_{L^1(\R^d)}\rho\right)
		+ \mathcal{O}(\e).
	\end{split}
\end{equation}

Sending $\delta\to 0$ and
using $\abs{\nabla\psi(\x)}\le C_0\psi(\x)$, we obtain
\begin{equation*}
	\begin{split}
		&\lim_{\delta\to 0} \abs{E\left[\iint  \!\!\int_0^t  \q_\rho^\f(u(s,\x),v(s,\y))
		\cdot \nabla\psi(\tfrac{\x+\y}{2})
		J_\delta(\x-\y)\,ds \,d\x\,d\y\right]}
		\\ &\qquad \qquad
		\le C_2 \|\f'\|_{L^\infty}\int_0^t
		E\left[\, \int \abs{u(s,\x)-v(s,\x)}\psi(\x) d\x\, \right]\, ds;
	\end{split}
\end{equation*}
hence, sending $\delta\to 0$ in \eqref{eq:uv-entropy2} returns
\begin{equation*}
	\begin{split}
		& E\left[ \int \abs{u(t,\x)-v(t,\x)} \psi(\x)\,d\x \right]
		-E\left[ \int \abs{u_0(\x)-v_0(\x)}\psi(\x)\,d\x \right]
		\\ & \le
		C_2\|\f'\|_{\infty} \int_0^t E\left[\, \int \abs{u(s,\x)-v(s,\x)}\psi(\x) d\x\, \right]\, ds
		\\ & \quad
		+ E\left[\int\!\!\int_0^t  \frac12 \eta_\rho''(u(s,\x)-v(s,\x))
		\big(\sigma(u(s,\x))-\sh(v(s,\x))\big)^2
		 \psi(\x)\, ds\,d\x\right]
		\\ & \qquad
		+ t\, E\left[ \abs{v_0}_{BV(\R^d)} \right] \norm{\psi}_{L^\infty(\R^d)}
		\left( \|\f'-\hf'\|_{L^\infty} +\mathcal{O}(\rho)\right)
		\\ & \qquad \quad
		+ \mathcal{O}\left(\norm{\psi}_{L^1(\R^d)}\rho\right)+ \mathcal{O}(\e).
	\end{split}
\end{equation*}

Next, with our choice of $\eta_\rho$, it follows that
\begin{equation}\label{eq:AplusB}
	\begin{split}
		&\abs{E\left[\int \! \! \int_0^t \frac12 \eta_\rho''(u(s,\x)-v(\tau,\x))
		\big(\s(u(s,\x))-\hs(v(s,\x))\big)^2\psi(\x)\, ds\,d\x\right]}
		\\ & \le E\left[ \int \! \! \int_0^t \frac{M_2}{\rho}
		\mathbf{1}_{\abs{u(s,\x)-v(s,\x)}<\rho}
		\big(\s(u(s,\x))-\hs(u(s,\x))\big)^2\psi(\x)\, ds\,d\x \right]
		\\ & \quad
		+E\left[\int \! \! \int_0^t \frac{M_2}{\rho}
		\mathbf{1}_{\abs{u(s,\x)-v(s,\x)}<\rho}
		\big(\hs(u(s,\x))-\hs(v(s,\x))\big)^2
		\psi(\x)\, ds\,d\x\right]
		\\ & =: A+B.
	\end{split}
\end{equation}
Clearly,
\begin{align*}
	\abs{A} &  \le C_3
	E\left[\int \! \! \int_0^t\frac{\abs{\s(u(s,\x))-\hs(u(s,\x))}^2}{\rho}
	\psi(\x)\,ds \,d\x\right] \\
         & \le C_3 \|\psi\|_{L^1(\R^d)} t
         \| \s-\hs \|_{L^\infty}^2 / \rho
\end{align*}
and, in view of \eqref{eq:sigma-ass},
\begin{align*}
	\abs{B}\le C_4
	\int_0^t E\left[\int \abs{u(s,\x)-v(s,\x)}\psi(\x)\,d\x\right]\, ds.
\end{align*}

In summary, we have arrived at
\begin{equation*}
	\begin{split}
		& E\left[ \int \abs{u(t,\x)-v(t,\x)} \psi(\x)\,d\x \right]
		-E\left[ \int \abs{u_0(\x)-v_0(\x)}\psi(\x)\,d\x \right]
		\\ & \quad \le
		C \Biggl( \|\f'\|_{L^\infty}\int_0^t
		E\left[\, \int \abs{u(s,\x)-v(s,\x)}\psi(\x) d\x\, \right]\, ds
		\\ &\quad \qquad \qquad
		+ \|\psi\|_{L^\infty(\R^d)} E\left[ \abs{v_0}_{BV(\R^d)} \right]  t
		\left( \|\f'-\hf'\|_{L^\infty}+\rho\right)
		\\ &  \quad \qquad \qquad\qquad
		+\|\psi\|_{L^1(\R^d)} t \|\s-\hs\|_{L^\infty}^2 / \rho
		 +\norm{\psi}_{L^1(\R^d)}\rho+\e\Biggr),
         \end{split}
\end{equation*}
which implies via the Gronwall inequality that, for any $t>0$,
\begin{equation}\label{eq:u-plus-v-entropy_final}
	\begin{split}
		&E\left[\int \abs{u(t,\x)-v(t,\x)}\,\psi(\x)\,d\x\right]
		\\ & \le e^{C\|\f'\|_{L^\infty} t}
		E\left[\int \abs{u_0(\x)-v_0(\x)}\psi(\x)\,d\x \right]
		\\ &\qquad + C e^{C\|\f'\|_{L^\infty} t}\Biggl(
		\|\psi\|_{L^\infty(\R^d)}
		E\left[ \abs{v_0}_{BV(\R^d)} \right]
		t \left( \|\f'-\hf'\|_{L^\infty} +\rho\right)
		\\ & \qquad\qquad \qquad \qquad\qquad
		+\|\psi\|_{L^1(\R^d)} t
		\|\s-\hs\|_{L^\infty}^2 / \rho
		+\norm{\psi}_{L^1(\R^d)}\rho+\e\Biggr).
	\end{split}
\end{equation}
Choosing $\rho= \sqrt{t} \|\s-\hs\|_{L^\infty}$ and
sending $\e\to 0$ supplies part (i).

\medskip
About part (ii), the only difference in the proof comes from
the estimate of the $A$-term
in \eqref{eq:AplusB}, which is replaced by
\begin{align*}
	|A| & \le C_3
	E\Big[\int \! \! \int_0^t
	\frac{|\s(u(s,\x))-\hs(u(s,\x))|^2}{\rho |u(s,\x)|^2} |u(s,\x)|^2\psi(\x)\, ds \, d\x\Big]
	\\ & = C_3E\Big[\int \! \! \int_0^t \frac{\big(\Delta(\s,\hs)\big)^2}{\rho}
	|u(s,\x)|^2\psi(\x) \, ds\, d\x\Big]
	\\ &\le C_3 \|\psi\|_{L^\infty(\R^d)} E\left [\|u\|_{L^\infty(0,T; L^2(\R^d))}\right]
	 \frac{t \, \big(\Delta(\sigma,\sh)\big)^2}{\rho}.
\end{align*}
With this estimate at our disposal, \eqref{eq:u-plus-v-entropy_final}
is replaced by
\begin{equation*}
	\begin{split}
		&E\left[\int \abs{u(t,\x)-v(t,\x)}\,\psi(\x)\,d\x\right]
		\\ & \le e^{C\|\f\|_{L^\infty} t}
		E\left[\int \abs{u_0(\x)-v_0(\x)}\psi(\x)\,d\x \right]
		\\ &\quad + C e^{C\|\f'\|_{L^\infty} t}\Biggl(
		\|\psi\|_{L^\infty(\R^d)}
		E\left[ \abs{v_0}_{BV(\R^d)} \right]
		t \left( \|\f'-\hf'\|_{L^\infty} + \rho\right)
		\\ &  \qquad\qquad\qquad\qquad\quad
		+\|\psi\|_{L^\infty(\R^d)} \frac{t \, \big(\Delta(\s,\hs)\big)^2}{\rho}
		+\norm{\psi}_{L^1(\R^d)}\rho+\e \Biggr).
	\end{split}
\end{equation*}
Part (ii) follows by choosing
$\rho= \sqrt{t}\Delta(\s, \hs)$ and sending $\e\to 0$.
\end{proof}

\begin{theorem}[Error estimate]\label{thm:errorest}
Suppose \eqref{1.3}--\eqref{eq:sigma-ass} hold.
Let $u(t)$ be the strong stochastic entropy
solutions of \eqref{eq:u}, for which
\begin{equation}\label{eq:errorest-BV}
	E\left[ \abs{u(t)}_{BV\R^d)}\right]\le
	\abs{u_0}_{BV\R^d)}
	\quad \text{for $t>0$},
\end{equation}
and let $\ueps$ be the
solution to the parabolic problem
\begin{equation*}
	d \ueps+ \bigl[ \nabla_\x\cdot \f(\ueps)-\e\Delta_\x \ueps\bigr]\, dt
	=\s(\ueps)\, dW(t),
	 \quad \ueps|_{t=0}=u_0.
\end{equation*}
In addition, we assume that
$$
\text{either $u,v\in L^\infty((0,T)\times\R^d\times\Omega)$ for any $T>0$,
or $\f''\in L^\infty$}.
$$
Then there exists a constant $C_T>0$ such that,
for any $0<t<T$ with $T$ finite,
$$
E\left[\int_{\R^d} |u(t,\x)-\ueps(t,\x)|\, d\x\right]
\le C_{T}\, E\left[\abs{u_0}_{BV(\R^d)}\right]
\, t  \, \sqrt{\e}.
$$
\end{theorem}

\begin{proof}
We proceed as in the proof of Theorem \ref{conti-dependence}, starting
off from Lemma \ref{lem:contdep} with $\hs=\s$, $\hf=\f$, $\he\neq\e$, $\ueps=u$,
$\uheps=v$, leading to
\begin{equation}\label{eq:uv-entropy2-error}
	\begin{split}
		&E\left[\iint \abs{\ueps(t,\x)-\uheps(t,\y)} \phi_\delta(\x,\y)\,d\x\, d\y\right]
		\\& \le E\left[\iint  \!\!\int_0^t  \q_\rho^\f(\ueps(s,\x),\uheps(s,\y))\cdot
		\nabla \psi(\tfrac{\x+\y}{2})
		J_\delta(\tfrac{\x-\y}{2})\,ds \,d\x\,d\y\right]
		\\ & \quad
		+ E\Biggl[\iint\!\!\int_0^t \eta_\rho''(\ueps(s,\x)-\uheps(s,\y))
		\\ & \qquad \qquad \qquad \quad
		\times \bigl(\s(\ueps(s,\x))-\s(\uheps(s,\y))\bigr)^2
		\phi_\delta(\x,\y)\, ds\,d\x\,d\y\Biggr]
		\\ & \quad
		+ t\,  \abs{u_0}_{BV(\R^d)} \norm{\psi}_{L^\infty(\R^d)}
		 \mathcal{O}(\rho)
		+ \mathcal{O}\left(\norm{\psi}_{L^1(\R^d)}\rho\right)
		\\ & \quad
		+ \big(\sqrt{\e}-\sqrt{\he}\big)^2
		E\Biggl[\iint \!\! \int_0^t
		\eta_\rho(\ueps(s,\x)-\uheps(s,\y))
		\\ & \qquad \qquad\qquad\qquad\qquad \qquad\quad\times
		\Delta_\y J_\delta(\tfrac{\x-\y}{2})
		\psi(\tfrac{\x+\y}{2}) \, ds\,d\x\,d\y\Biggr]
		\\ & \quad + \frac14 \big(\sqrt{\e}+\sqrt{\he}\big)^2
		E\Biggl[\iint \!\! \int_0^t
		\eta_\rho(\ueps(s,\x)-\uheps(s,\y))
		\\ & \qquad \qquad\qquad\qquad\qquad \qquad\quad\times
		J_\delta(\tfrac{\x-\y}{2})
		\Delta\psi(\tfrac{\x+\y}{2})  \, ds\,d\x\,d\y\Biggr]
		\\ & \quad
		+\big(\he-\e \big) E\Biggl[\iint \!\! \int_0^t
		\eta_\rho(\ueps(s,\x)-\uheps(s,\y))
		\\ & \qquad \qquad\qquad\qquad\qquad\qquad\quad \times
		\nabla_\y J_\delta(\tfrac{\x-\y}{2})
		\cdot \nabla \psi(\tfrac{\x+\y}{2})\, ds\,d\x\,d\y\Biggr]
		\\ &  =: I_1+I_2+I_3+I_4+I_5+I_6.
	\end{split}
\end{equation}

As before,
$$
\abs{I_2} \le C_1 \int_0^t E\left[\int \abs{\ueps(s,\x)-\ueps(s,\y)}
J_\delta(x-y)\psi(\tfrac{\x+\y}{2}) \,d\x \, d\y\right]\, ds.
$$
Noting that the right-hand side is independent
of $\rho$, we can first send $\rho\to 0$
in \eqref{eq:uv-entropy2-error}, and then let $\psi$ tend
to $\mathbf{1}_{\R^d}$, keeping in mind the $L^p$--estimates
\eqref{eq:Lp-BV}, with the outcome that $I_1,I_3,I_5,I_6\to 0$.
The resulting estimate reads
\begin{equation}\label{eq:error-1}
	\begin{split}
		&E\left[\iint \abs{\ueps(t,\x)-\uheps(t,\y)}
		J_\delta(\tfrac{\x-\y}{2})\,d\x\, d\y\right]
		\\ & \qquad \le
		C_1 \int_0^t E\left[\iint \abs{u(s,\x)-v(s,\y)}
		J_\delta(\tfrac{\x-\y}{2})\,d\x\right]\, ds+ I,
	\end{split}
\end{equation}
where
$$
I = \big(\sqrt{\e}-\sqrt{\he}\big)^2
E\left[\iint \!\! \int_0^t
|\ueps(s,\x)-\uheps(s,\y)|  \Delta_\y J_\delta(\tfrac{\x-\y}{2})
\, ds\,d\x\,d\y\right].
$$

An integration by parts, followed by application of
the spatial $BV$--estimate \eqref{eq:errorest-BV}, yields
$$
\abs{I} \le  C\,_2 \, t\,  E\left[ \abs{u_0}_{BV(\R^d)} \right]
\frac{\big(\sqrt{\e}-\sqrt{\he}\big)^2}{\delta}.
$$
In view of this, it follows from \eqref{eq:error-1} in a completely
standard way that
\begin{align*}
	& E\left[\int |\ueps(t,\x)-\uheps(t,\x)|\,d\x\right]
	\\ &\quad
	\le C_1 \int_0^t E\left[\int \abs{u^\e(s,\x)-v^\e(s,\x)}\,d\x\right]\, ds
	\\ & \quad \qquad\qquad
	+ C_3 E\left[ \abs{u_0}_{BV(\R^d)} \right]
	\Big(\delta + t  \frac{\big(\sqrt{\e}-\sqrt{\he}\big)^2}{\delta} \Big).
\end{align*}

Choosing $\delta=\sqrt{\e}-\sqrt{\he}$ gives
$$
E\left[\int_{\R^d} |\ueps(t,\x)-\uheps(t,\x)|\, d\x\right]
\le C_{T}\, E\left[ \abs{u_0}_{BV(\R^d)} \right]\,
t  \left(\sqrt{\e}-\sqrt{\he}\right).
$$
Sending $\he\to 0$ concludes the proof of the theorem.
\end{proof}

\begin{remark}
Theorem 5.2 indicates that $\{\ueps(t,\x)\}$ is the
Cauchy sequence in $C(0, T; L^1)$,
which directly implies its strong convergence.
\end{remark}

\section{More General Equations}\label{sec:general}

We now discuss briefly diverse generalizations.

First of all, as in \cite{FN}, the stochastic term in  \eqref{1.1}
can be replaced by the more general term
$$
\int_{z\in Z}\sigma(u(t,x);z)\partial_t W(t,dz),
$$
where $Z$ is a
metric space, $\sigma: \R\times Z\to \R$, $W(t,dz)$ is a space-time
Gaussian white noise martingale random measure
with respect to a filtration $\{\mathcal{F}_t\}$ (see e.g., Walsh \cite{Walsh},
Kurtz-Protter \cite{KurtzProtter}) with
$$
E\big[W(t,A)\cap W(t,B)\big]=\mu(A\cap B) t
$$
for measurable $A,B\subset Z$, where $\mu$ is a (deterministic) $\sigma$-finite Borel
measure on the metric space $Z$. In particular, when
$Z=\{1, 2,\dots, m\}$ and $\mu$ is a counting measure on $Z$,
then the stochastic term reduces to
$$
\sum_{k=1}^m\sigma_k(u(t,\x))\partial_t W_k(t).
$$

For the spatial $BV$ and temporal $L^1$--continuity estimates and
stability results, we can allow for more general
flux functions $\f(t,\x,u)$ with spatial dependence, by
combining the present methods with
those in \cite{Chen:2005wf,Karlsen:2003za}.

\medskip
Next, let us discuss the case where the noise coefficient $\sigma(\x,u)$
has a  spatial dependence, focusing on the stochastic balance law
\begin{equation}\label{6.1a}
	\partial_t u+ \nabla\cdot \f(u)
	=\s(\x,u)\, \partial_t W(t),
\end{equation}
where the noise coefficient is assumed to satisfy $\sigma(\x,0)=0$ and
\begin{equation}\label{eq:sigma-ass-xdep}
	\begin{split}
		& \abs{\s(\x,u)-\s(\x,v)} \leq C \abs{u-v},
		\qquad\quad\,\,\, \forall\, u,v\in \R,\, \forall\, \x\in \R^d,
		\\ & \abs{\s(\x,u)-\s(\y,u)}\leq C \abs{\x-\y} \abs{u},
		\qquad \forall \, u\in \R,\, \forall\, \x,\y\in \R^d,
	\end{split}
\end{equation}
where $C$ is a deterministic constant.

In the previous sections, we have established the existence of a
strong stochastic entropy solution in the multidimensional context.
The proof was based on deriving $BV$--estimates. However, as mentioned
before, the $BV$--estimates are no longer available when
the noise term $\s$ depends on the spatial location $\x$.
However, it possible to derive fractional $BV$
estimates. For fixed $\e>0$, let $\ueps(t,\x)$ be the
solution to the stochastic parabolic problem
\begin{equation}\label{eq:pareqn}
	d \ueps+ \bigl[ \nabla_\x\cdot \f(\ueps)-\e\Delta_\x \ueps\bigr]\, dt
	=\s(\x,\ueps)\, dW(t),
	 \quad \ueps|_{t=0}=u_0,
\end{equation}
where we tactically assume that $\f,\sigma, u_0$ are sufficiently smooth
to ensure the existence of a regular solution \cite{FN}.
Utilizing the continuous dependence
framework (Lemma \ref{lem:contdep}) which also holds when the noise
term $\sigma$ depends on $\x$,  we will prove that,
for any $\delta>0$,
\begin{equation}\label{eq:fractional-ueps}
	\begin{split}
		&E\left[ \int_{\R^d}\int_{\R^d} \abs{\ueps(t,\x+\z)-\ueps(t,\x-\z)}
		J_\delta(\z) \psi(\x)\,d\x\, d\z\right]
		\\ & \quad \le  C_T \,
		E\left[\int_{\R^d}\int_{\R^d} \abs{u_0(\x+\z)-u_0(\x-\z)}
		J_\delta(\z) \psi(\x)\,d\x\,d\z\right]
		\\ & \quad\qquad
		+ C_T\, \delta^{\frac12}\, \left(1+\norm{\psi}_{L^1(\R^d)}\right),
		\qquad 0<t<T,
	\end{split}
\end{equation}
for some finite constant $C_T$ independent of $\eps$, where
$J_\delta$ is a symmetric mollifier and $\psi\ge 0$ is a
compactly supported smooth function.
In what follows, we assume that the
cut-off function $\psi\ge 0$ satisfies
\begin{equation*}
	\begin{split}
		& \abs{\nabla\psi(\x)}\le C_0\psi(\x),
		\quad \abs{\Delta \psi(\x)}\le C_0\psi(\x),
		\quad \text{$\psi \equiv 1$ on $K_R:=\Set{\abs{\x}<R}$},
	\end{split}		
\end{equation*}
for some constants $C_0>0$ and $R>0$.
One example of such a function, at least after an easy approximation argument, is the
compactly supported function $\psi\in W^{2,\infty}(\R^d)$ defined by
$$
\psi(\x)
=\begin{cases}
1 \qquad &\text{when $|\x|\le R$},\\
\frac{1}{e^{\pi}+1}\big(\sqrt{2}e^{\pi-(|\x|-R)}\sin (|\x|-R+\frac{\pi}{4}) + 1\big) \qquad &\text{when
$R\le |\x|\le R+\pi$},\\
0 \qquad &\text{when $|\x|\ge R+\pi$}.
\end{cases}
$$

\medskip
Estimate \eqref{eq:fractional-ueps} can be turned into
a fractional $BV$ estimate thanks to the following
deterministic lemma, which is related to known
links between Sobolev, Besov, and
Nikolskii fractional spaces (cf., e.g., \cite{Simon:1990fk}); a proof
can be found in the appendix.

\begin{lemma}\label{lem:fractional}
Let $h:\R^d\to \R$ be a given integrable
function, $r,s\in (0,1)$, $\psi\in C^\infty_c(\R^d)$,
and $\Set{J_\delta}_{\delta>0}$ a sequence
of symmetric mollifiers, i.e., $J_\delta(x)
=\tfrac{1}{\delta^d}J\left(\tfrac{\abs{x}}{\delta}\right)$,
$0\le J\in C^\infty_c(\R)$, $\text{supp}\, (J)\subset [-1,1]$,
$J(-\cdot)=J(\cdot)$, and $\int J=1$.

Suppose $r<s$. Then there
exists a finite constant $C_1=C_1(J,d,r,s)$ such that,
for any $\delta>0$
\begin{equation}\label{eq:sob-to-trans}
	\begin{split}
		&\int_{\R^d}\int_{\R^d} \abs{h(\x+\z)-h(\x-\z)}
		J_\delta(\z)\psi(\x) \,d\x \, d\z
		\\ &
		\le C_1\, \delta^{r}\, \sup_{\abs{\z}\le \delta}\,
		\abs{\z}^{-s} \int_{\R^d}\abs{h(\x+\z)-h(\x-\z)}\psi(\x)\, d\x.
	\end{split}
\end{equation}

Suppose $r<s$. Then there exists
a finite constant $C_2=C_2(J,d,r,s)$
such that for any $\delta>0$
\begin{equation}\label{eq:trans-to-sob}
	\begin{split}
		& \sup_{\abs{z}\le \delta}\,
		\int_{\R^d}\abs{h(\x+\z)-h(\x)}\psi(\x)\, d\x
		\\ & \,
		\le C_2\, \delta^r\, \sup_{0<\delta\le 1}\,  \delta^{-s}
		\int_{\R^d}\int_{\R^d} \abs{h(\x+\z)-h(\x-\z)}
		J_\delta(\z)\psi(\x)\,d\x\,d\z
		\\ & \qquad
		+C_2\, \delta^r\,\norm{h}_{L^1(\R^d)}.
	\end{split}
\end{equation}
\end{lemma}

Suppose $u_0$ is say a deterministic
function belonging to $BV(\R^d)$,
or more generally to the Besov space $B_{1,\nu}^\ell(\R^d)$ for
$\nu\in (\frac12,1)$.

Starting off from \eqref{eq:fractional-ueps} with $\delta>0$,
\begin{equation}\label{eq:fractional-tmp1}
	\begin{split}
		& \delta^{-\frac12}E\left[\int_{\R^d}
		\int_{\R^d} \abs{\ueps(t,\x+\z)-\ueps(t,\x-\z)}
		J_\delta(\z)\psi(\x)\,d\x\,d\z\right]
		\\ & \quad \le C_T\, \delta^{-\frac12} \,
		\int_{\R^d}\int_{\R^d} \abs{u_0(\x+\z)-u_0(\x-\z)}
		J_\delta(\z)\,\psi(\x)\,d\x\,d\z
		\\ & \quad\qquad
		+ C_T\, \left(1+\norm{\psi}_{L^1(\R^d)}\right)
		\\ & \quad \le 2\, C_T\, C_1 \, \norm{\psi}_{L^\infty(\R^d)}\,
		\sup_{\abs{\z}\le \delta }\,
		\abs{z}^{-s} \int_{\R^d}\abs{u_0(\x+\z)-u_0(\x)} \,d\x
		\\ & \quad\qquad
		+ C_T\, \left(1+\norm{\psi}_{L^1(\R^d)}\right)\\
		&\quad \le C(T,R),
	\end{split}
\end{equation}
where \eqref{eq:sob-to-trans}
with $r=\frac12$ and $s>\tfrac12$ was used
to arrive at the second inequality.

In view of \eqref{eq:trans-to-sob} with $s=\tfrac12$ and $r<\tfrac12$,
\begin{equation}\label{eq:fractional-tmp2}
	\begin{split}
		&\sup_{\abs{\z}\le \frac{\delta}{2}}E\left[\int_{\R^d}
		\abs{\ueps(t,\x+\z)-\ueps(t,\x)}\psi(\x)\,d\x\right]
		\\ & \quad
		\le C_2\, \delta^r\,\sup_{0<\delta\le 1}\,  \delta^{-\frac12}
		\int_{\R^d}\int_{\R^d} \abs{\ueps(t,\x+\z)-\ueps(t,\x-\z)}
		J_{\delta}(\z)\psi(\x)\,d\x\,d\z
		\\ & \quad\qquad
		+C_2\, \delta^r\, \norm{\ueps(t,\cdot)}_{L^1(\R^d)}.
	\end{split}
\end{equation}

Combining \eqref{eq:fractional-tmp1}
with \eqref{eq:fractional-tmp2}  yields

\begin{theorem}[Fractional $BV$ estimate]
For fixed $\eps>0$, let $\ueps$ solve the stochastic
parabolic problem \eqref{eq:pareqn} with deterministic
initial data $u_0$ belonging to the Besov
space $B_{1,\infty}^\nu(\R^d)$ for
some $\nu\in (\frac12,1)$.  In addition, we assume that
$$
\text{either $u^\e\in L^\infty((0,T)\times\R^d\times\Omega)$ for any $T>0$,
or \, $\f''\in L^\infty$}.
$$
Fix $T>0$ and $R>0$. There exists
a constant $C_{T,R}$ independent
of $\eps$ such that, for any $0<t<T$,
\begin{align*}
	\sup_{\abs{\z}\le \delta}
	E\left[\int_{K_R} \abs{\ueps(t,\x+\z)-\ueps(t,\x)}\,d\x\right] \le C_{T,R} \, \delta^r
\end{align*}
for some $r\in (0,\tfrac12)$.
\end{theorem}

\begin{proof}[Proof of \eqref{eq:fractional-ueps}]
We start off from Lemma \ref{lem:contdep} with $\hf=\f$, $\he=\e$,
$\hs=\s$, $v_0=u_0$, $v=u$ (which also holds
when $\s$ depends on the spatial location):
\begin{equation}\label{eq:uv-entropy1-fBV-a}
	\begin{split}
		&E\left[\iint \eta_\rho(\ueps(t,\x)-\ueps(t,\y)) J_\delta(\tfrac{\x-\y}{2})
		\psi(\tfrac{\x+\y}{2}) \,d\x\, d\y\right]
		\\ & \qquad
		-E\left[\iint \eta_\rho(u_0(\x)-u_0(\y)) J_\delta(\tfrac{\x-\y}{2})
		\psi(\tfrac{\x+\y}{2}) \,d\x\,d\y\right]
		\\& \le E\left[\iint  \!\!\int_0^t  \q^\f_\rho(\ueps(s,\x),\ueps(s,\y))\cdot
		\nabla \psi(\tfrac{\x+\y}{2}) J_\delta(\tfrac{\x-\y}{2})\,ds \,d\x\,d\y\right]
		\\ & \;\;
		+ E\left[\iint  \!\! \int_0^t \bigl(\q^{\f}_\rho(\ueps(s,\y),\ueps(s,\x))
		-\q^\f_\rho(\ueps(s,\x),\ueps(s,\y))\bigr)\cdot \nabla_\y\phi_\delta
		\, ds\,d\x\,d\y\right]
		\\ & \;\;
		+ E\Biggl[\iint\!\!\int_0^t \frac12 \eta_\rho''(\ueps(s,\x)-\ueps(s,\y))
		\\ & \qquad\qquad\qquad\qquad \times
		\bigl(\s(\x,\ueps(s,\x))-\s(\y,\ueps(s,\y))\bigr)^2\phi_\delta(\x,\y)\, ds\,d\x\,d\y\Biggr]
		\\ & \;\; +\, \eps\, E\Biggl[\iint \!\! \int_0^t \eta_\rho(\ueps(s,\x)-\ueps(s,\y))
		J_\delta(\tfrac{\x-\y}{2})\Delta_\x\psi(\tfrac{\x+\y}{2})  \, ds\,d\x\,d\y\Biggr]
		\\ & =: I_1+I_2+I_3+I_4.
	\end{split}
\end{equation}

Finally, denoting the left-hand
side of \eqref{eq:uv-entropy1-fBV-a} by \text{LHS} and
utilizing \eqref{2.4b}, we have
\begin{align*}
	\text{LHS} & = E\left[\iint \abs{\ueps(t,\x)-\ueps(t,\y)}
	J_\delta(\tfrac{\x-\y}{2}) \psi(\tfrac{\x+\y}{2})
	\,d\x\, d\y\right]
	\\ & \quad -E\left[\iint \abs{u_0(\x)-u_0(\y)}
	J_\delta(\tfrac{\x-\y}{2}) \psi(\tfrac{\x+\y}{2})\,d\x\,d\y\right]
	+ \mathcal{O}(\rho) \norm{\psi}_{L^1(\R^d)}.
\end{align*}

Since $\abs{\nabla\psi(\x)}\le C_0\psi(\x)$,
\begin{align*}
	\abs{I_1}\le C
	\int_0^t E\left[\iint \abs{\ueps(s,\x)-\ueps(s,\y)}
	J_\delta(\tfrac{\x-\y}{2}) \psi(\tfrac{\x+\y}{2})
	\,d\x\, d\y\right]\, ds.
\end{align*}

Note that, thanks to \eqref{eq:entropy-prop} and the boundedness of $\f''$,
$$
\q_\rho^\f(v,u)=\q_\rho^\f(u,v)+\int_v^u \partial_\xi
\left(\q_\rho^\f(\xi,v)-\q_\rho^\f(v,\xi)\right)\,d\xi=\q_\rho^\f(u,v)
+\abs{u-v}O(\rho),
$$
so that
\begin{align*}
	\abs{I_2} &\le C \, \rho \,
	 E\left[\iint  \!\! \int_0^t \abs{\ueps(s,\x)-\ueps(s,\y)}
	 \abs{\nabla_\y J_\delta(\tfrac{\x-\y}{2})}
	 \psi(\tfrac{\x+\y}{2}) \, ds\,d\x\,d\y\right]
	 \\ & \qquad
	 + C \, \rho\, E\left[\iint  \!\! \int_0^t \abs{\ueps(s,\x)-\ueps(s,\y)}
	 J_\delta(\tfrac{\x-\y}{2}) \abs{\nabla \psi(\tfrac{\x+\y}{2})}
	 \, ds\,d\x\,d\y\right]
	 \\ & \le C \, t \,  \norm{\psi}_{L^\infty(\R^d)}
	 \left(\frac{\rho}{\delta}+\rho \right),
\end{align*}
because of the estimate
$$
\sup_{0\le t\le T} E\left[\norm{\ueps(t)}_{L^1(\R^d)}\right]<\infty,
\qquad \text{for any $T>0$},
$$
and we have again exploited
$\abs{\nabla\psi(\x)}\le C_0\psi(\x)$.

Regarding $I_3$,
\begin{equation*}
	\begin{split}
		\abs{I_3} & \le
		E\Biggl[ \iint \! \! \int_0^t \frac{M_2}{\rho}
		\mathbf{1}_{\abs{\ueps(s,\x)-\ueps(s,\x)}<\rho}
		\bigl(\s(\x,\ueps(s,\x))-\s(\y,\ueps(s,\x))\bigr)^2
		\\ & \qquad\qquad \qquad\qquad\qquad\qquad
		\times J_\delta(\tfrac{\x-\y}{2})
		\psi(\tfrac{\x+\y}{2})\, ds\,d\x\,d\y \Biggr]
		\\ & \qquad
		+E\Biggl[\iint \! \! \int_0^t \frac{M_2}{\rho}
		\mathbf{1}_{\abs{\ueps(s,\x)-\ueps(s,\y)}<\rho}
		\bigl(\s(\y,\ueps(s,\x))-\s(\y,\ueps(s,\y))\bigr)^2
		\\ & \qquad\qquad \qquad\qquad\qquad\qquad
		\times J_\delta(\tfrac{\x-\y}{2})
		\psi(\tfrac{\x+\y}{2} )\, ds\,d\x\,d\y \Biggr]=: A+B,
	\end{split}
\end{equation*}
where, cf.~second part of \eqref{eq:sigma-ass-xdep},
\begin{align*}
	\abs{A} &  \le M_2
	E\left[\iint \! \!
	\int_0^t \frac{\abs{\s(\x,\ueps(s,\x))-\s(\y,\ueps(s,\x))}^2}{\rho}
	J_\delta(\tfrac{\x-\y}{2})
	\psi(\tfrac{\x+\y}{2})\,ds \,d\x\, d\y\right] \\
	&  \le C
	E\left[\int \! \! \iint_0^t \frac{|\y-\x|^2}{\rho} |\ueps(s,\x)|^2
	J_\delta(\tfrac{\x-\y}{2})\psi(\tfrac{\x+\y}{2})\,ds \,d\x\, d\y\right] \\
         & \le C \norm{\psi}_{L^\infty(\R^d)}\, t\,  \frac{\delta^2}{\rho},
\end{align*}
where we have put to use the estimate
$$
\sup_{0\le t\le T} E\left[\norm{\ueps(t)}_{L^2(\R^d)}^2\right]<\infty
\qquad \text{for any $T>0$}.
$$
Moreover, cf.~first part of \eqref{eq:sigma-ass-xdep},
\begin{align*}
	\abs{B}\le C
	\int_0^t E\left[\iint \abs{\ueps(s,\x)-\ueps(s,\y)}
	J_\delta(\tfrac{\x-\y}{2}) \psi(\tfrac{\x+\y}{2})
	\,d\x\, d\y\right]\, ds.
\end{align*}

Regarding $I_4$, using $\abs{\Delta \psi(\x)}\le C_0\psi(\x)$, we have
\begin{align*}
	\abs{I_4}\le C
	\int_0^t E\left[\iint \abs{\ueps(s,\x)-\ueps(s,\y)}
	J_\delta(\tfrac{\x-\y}{2}) \psi(\tfrac{\x+\y}{2})
	\,d\x\, d\y\right]\, ds.
\end{align*}

Summarizing, we have arrived at
\begin{equation*}
	\begin{split}
		&E\left[\iint \abs{\ueps(t,\x)-\ueps(t,\y)}
		J_\delta(\tfrac{\x-\y}{2}) \psi(\tfrac{\x+\y}{2})
		\,d\x\, d\y\right]
		\\ & \le  E\left[\iint \abs{u_0(\x)-u_0(\y)}
		J_\delta(\tfrac{\x-\y}{2}) \psi(\tfrac{\x+\y}{2})\,d\x\,d\y\right]
		\\ & \quad
		+ C\int_0^t E\left[\iint \abs{\ueps(s,\x)-\ueps(s,\y)}
		J_\delta(\tfrac{\x-\y}{2})\psi(\tfrac{\x+\y}{2})\,d\x\, d\y\right]\, ds
		\\&\quad\quad
		 +C \, t \,  \norm{\psi}_{L^\infty(\R^d)}
		 \left(\frac{\rho}{\delta}+\rho \right)
		 + C \norm{\psi}_{L^\infty(\R^d)}\, t\,  \frac{\delta^2}{\rho}
		 +C\rho \norm{\psi}_{L^1(\R^d)}.
	\end{split}
\end{equation*}
Optimizing with respect to $\rho$ (take $\rho=\mathcal{O}(\delta^{3/2})$)
and applying  Gronwall's lemma gives
\begin{equation}\label{eq:uv-entropy1-fBV-c}
	\begin{split}
		&E\left[\iint \abs{\ueps(t,\x)-\ueps(t,\y)}
		J_\delta(\tfrac{\x-\y}{2}) \psi(\tfrac{\x+\y}{2})
		\,d\x\, d\y\right]
		\\ & \quad \le  C_T E\left[\iint \abs{u_0(\x)-u_0(\y)}
		J_\delta(\tfrac{\x-\y}{2}) \psi(\tfrac{\x+\y}{2})\,d\x\,d\y\right]
		\\ & \quad\qquad \qquad
		+ C_T\left(1+\norm{\psi}_{L^1(\R^d)}\right) \, \sqrt{\delta},
		\qquad 0<t<T,
	\end{split}
\end{equation}
for some constant $C_T$ independent of $\eps$.

Introducing new variables, $\tilde{\x}=\frac{\x+\y}{2}$
and $z=\frac{\x-\y}{2}$ in \eqref{eq:uv-entropy1-fBV-c}, so
$\x=\tilde{\x}+\z$ and $\y=\tilde{\x}+\z$,  we finally
obtain (dropping the tildes) \eqref{eq:fractional-ueps}.
\end{proof}

Combining Theorem 6.2 with the argument in Section 3, we conclude

\begin{theorem}[Existence and regularity]  Let \eqref{eq:sigma-ass-xdep}
and $\|\f''\|_{L^\infty}<\infty$ hold.

\begin{enumerate}
\item[(i)] Let the initial data $u_0$ belong to the Besov
space $B_{1,\infty}^\nu(\R^d)$ for
some $\nu\in (\frac12,1)$ and
\begin{equation}\label{6.7a}
E\left[\\u_0\|_{L^p(\R^d)}^p\right]<\infty, \quad p=1,2, \cdots.
\end{equation}
Then there exists a strong stochastic entropy solution of the balance law \eqref{6.1a}
with initial data $u_0$ such that, for fixed $T>0$ and $R>0$, there exists
a constant $C_{T,R}$ such that, for any $0<t<T$,
\begin{align*}
	\sup_{\abs{\z}\le \delta}
	E\left[\int_{K_R} \abs{u(t,\x+\z)-u(t,\x)}\,d\x\right] \le C_{T,R} \, \delta^r
\end{align*}
for some $r\in (0,\tfrac12)$ and
\begin{equation}\label{6.8a}
E\left[\|u(t,\cdot)\|_{L^p(\R^d)}^p\right]<\infty, \quad p=1,2, \cdots.
\end{equation}

\item[(ii)] Let $u_0$ satisfy only \eqref{6.7a}.
Then there exists a strong stochastic entropy solution of the balance law \eqref{6.1a}
with initial data $u_0$ satisfying \eqref{6.8a}.
\end{enumerate}
\end{theorem}

\medskip
Finally, we remark in passing that the results and techniques
straightforward extends to nonlinear stochastic balance laws with additional
nonhomogeneous terms, by combining with the Gronwall inequality, such as
\begin{equation*}\label{1.1b}
	\partial_t u(t,\x)+ \nabla\cdot \f(u(t,\x))
	=\s(u(t,\x))\, \partial_t W(t) +g(u(t,\x), t,\x), \quad \x\in \R^d,\, t>0,
\end{equation*}
with initial data \eqref{1.2},
for a large class of nonhomogeneous terms $g(u,t,\x)$.

\appendix

\section{Proof of Lemma \ref{lem:fractional}}

Suppose $r<s$, and let us prove \eqref{eq:sob-to-trans} as follows:
\begin{align*}
	& \delta^{-r}\int_{\R^d} \int_{\R^d} \abs{h(\x+\z)-h(\x-\z)}
	J_\delta(\z)\psi(\x)\,d\z \, d\x
	\\ &\quad = \int_{\R^d} \int_{\R^d} \frac{\abs{h(\x+\z)-h(\x-\z)}}{\delta^{d+r}}
	J(\tfrac{\abs{\z}}{\delta}) \psi(\x)\, d\z\,d\x
	\\ &\quad \le \norm{J}_{L^\infty(\R)}
	\int_{\R^d} \int_{\abs{z}\le \delta} \frac{\abs{h(\x+\z)-h(\x-\z)}}{\abs{\z}^{d+r}}
	\psi(x)\, d\z\,d\x.
	\\ & \quad \le  \norm{J}_{L^\infty(\R)}
	\sup_{\abs{\z}\le \delta }\, \z^{-s}
	\norm{\left(h(\cdot+\z)-h(\cdot-\z)\right)\psi}_{L^1(\R^d)}
	\int_{\abs{\z}\le \delta}\frac{1}{\abs{\z}^{d+r-s}}\,d\z
	\\ &\quad \le  C_{J,d,r,s}
	\sup_{\abs{\z}\le \delta }\, \z^{-s}
	\norm{\left(h(\cdot+\z)-h(\cdot-\z)\right)\psi}_{L^1(\R^d)},
\end{align*}
where we have used the integrability of $1/\abs{z}^{d+r-s}$ (since $d + r- s <d$).

We continue with the proof of \eqref{eq:trans-to-sob}.
To this end, let us introduce the modulus of continuity
$$
\omega(\delta):=\sup_{\abs{\z}\le \delta}
\int_{\R^d} \abs{h(\x+\z)-h(\x)}\psi(\x)\, d\x, \quad \delta>0.
$$

Clearly, $\omega(\cdot)$ is a non-decreasing function and thus
\begin{align*}
	&\int_0^\infty \kappa^{-r-1}\omega(\kappa)\, d\kappa
	\ge \int_{\delta}^\infty \kappa^{-r-1}\omega(\kappa)\, d\kappa
	\ge \omega(\delta) \int_{\delta}^\infty \kappa^{-r-1}\, d\kappa
	= \frac{1}{r}\delta^{-r}\omega(\delta);
\end{align*}
therefore
\begin{equation}\label{eq:mod-tmp1}
	\omega(\delta) \le  r \, \delta^{r}\,
	\int_0^\infty \kappa^{-r-1}\omega(\kappa)\, d\kappa.
\end{equation}

Set
$$
h_\delta(\x):=\int_{\R^d} J_{\frac{\delta}{2}}(\y)h(\x+\y)\,d\y,
$$
and note that
\begin{equation}\label{eq:h-hdelta}
	\begin{split}
	&\int_{\R^d} \abs{h(\x+\z)-h(\x)}\psi(\x)\, d\x
	\\ & \quad \le \int_{\R^d} \abs{h_\delta(\x+\z)-h_\delta(\x)}\psi(\x)\, d\x
	+\int_{\R^d} \abs{h_\delta(\x+\z)-h(\x+\z)}\psi(\x)\, d\x
	\\ & \quad\qquad
	+\int_{\R^d} \abs{h_\delta(\x)-h(\x)}\psi(\x)\, d\x,
	\end{split}
\end{equation}
We estimate the first two terms on the right-hand side as follows:
\begin{align*}
	&\int_{\R^d} \abs{h_\delta(\x)-h(\x)}\psi(\x)\, d\x
	\\ & \quad = \int_{\R^d}
	\abs{2^d\delta^{-d} \int_{\R^d} J(\tfrac{2\abs{\y}}{\delta})
	\left(h(\x+\y)-h(\x)\right)\,d\y}\psi(\x)\,d\x
	\\ & \quad \le \norm{J}_{L^\infty(\R)} \delta^{-d}
	\int_{\abs{\y}\le \frac{\delta}{2}}\int_{\R^d} \abs{h(\x+\y)-h(\x)}\psi(\x)\,d\x\,d\y
\end{align*}
and, similarly,
\begin{align*}
	&\int_{\R^d} \abs{h_\delta(\x+\z)-h(\x+\z)}\psi(\x)\, d\x
	\\ & \, \le \norm{J}_{L^\infty(\R)} \delta^{-d}
	\int_{\abs{\y}\le \frac{\delta}{2}}\int_{\R^d} \abs{h(\x+\z+\y)-h(\x+\z)}\psi(\x)\,d\x\,d\y
	\\ & \, = \norm{J}_{L^\infty(\R)} \delta^{-d}
	\int_{\abs{\y}\le \frac{\delta}{2}}\int_{\R^d} \abs{h(\x+\y)-h(\x)}\psi(\x-\z)\,d\x\,d\y
	\\ & \, \le  C\, \delta^{-d}
	\int_{\abs{\y}\le \frac{\delta}{2}}
	\int_{\R^d} \abs{h(\x+\y)-h(\x)}\psi(\x)\,d\x\,d\y + I_1(\delta),
\end{align*}
where, for $\delta\ge 0$,
\begin{align*}
	I_1(\delta) & := \delta^{-d} \sup_{\abs{\z}\le \frac{\delta}{2}}
	\int_{\abs{\y}\le \delta}\int_{\R^d} \abs{h(\x+\y)
	-h(\x)}\abs{\psi(\x)-\psi(\x-\z)}\,d\x\,d\y
	\\ & \le
	\delta\, C \norm{\nabla \psi}_{L^\infty(\R^d)} \norm{h}_{L^1(\R^d)}
	\mathbf{1}_{0\le \delta \le 1}(\delta)
	\\ & \qquad
	+ C \norm{\psi}_{L^\infty(\R^d)} \norm{h}_{L^1(\R^d)}
	\mathbf{1}_{\delta>1}(\delta).
\end{align*}

For each $\z\in\R^d$ and $\x\in\R^d$,
$$
h_\delta(\x+\z)-h_\delta(\x)=
\int_0^1 \nabla h_\delta(\x+\ell \z)\cdot \z \, d\ell.
$$
Observe that for each $x\in\R^d$,
\begin{align*}
	\nabla h_\delta(\x) = \int_{\R^d}
	\nabla J_{\frac{\delta}{2}}(\y) \left(h(\x+\y)-h(\x)\right)\,d\y.
\end{align*}
by the symmetry of the mollifier. Thus, with $\abs{\z}\le \delta$,
\begin{align*}
	& \int_{\R^d}  \abs{h_\delta(\x+\z)-h_\delta(\x)}\psi(\x)\, d\x
	\\ & = \int_{\R^d}  \abs{ \int_0^1
	\nabla h_\delta(\x+\ell \z)\cdot \z \, d\ell}\psi(\x)\, d\x
	\\ & \le C\, \delta^{-d} \sup_{\abs{\z} \le \delta, \,\ell\in[0,1]}
	\int_{\abs{y}\le \frac{\delta}{2}}
	\int_{\R^d} \abs{h(\x+\ell \z+\y)-h(\x+\ell \z)}\psi(\x)\,d\x\, d\y
	\\ & = C\,\delta^{-d}\,  \sup_{\abs{\z} \le \delta, \,\ell\in[0,1]}
	\int_{\abs{y}\le  \frac{\delta}{2}}
	\int_{\R^d} \abs{h(\x+\y)-h(\x)}\psi(\x-\ell \z) \,d\x\, d\y
	\\ & \le C\,
	\delta^{-d}\int_{\abs{y}\le  \frac{\delta}{2}}
	\int_{\R^d} \abs{h(\x+\y)-h(\x)}\psi(\x) \,d\x\, d\y+I_2(\delta),
\end{align*}
where $I_2(\delta)$ denotes the expression
\begin{align*}
	& C\, \delta^{-d} \, \sup_{\abs{\z} \le \delta, \,\ell\in[0,1]}\,
	\int_{\abs{y}\le \frac{\delta}{2}}
	\int_{\R^d} \abs{h(\x+\y)-h(\x)}\abs{\psi(\x)-\psi(\x-\ell \z)}\,d\x\, d\y,
\end{align*}
and
\begin{align*}
	I_2(\delta) & \le
	\delta\, C \norm{\nabla \psi}_{L^\infty(\R^d)} \norm{h}_{L^1(\R^d)}
	\mathbf{1}_{0\le \delta \le 1}(\delta)
	\\ & \qquad
	+ C \norm{\psi}_{L^\infty(\R^d)} \norm{h}_{L^1(\R^d)}
	\mathbf{1}_{\delta>1}(\delta),
\end{align*}
cf.~the term $I_1(\delta)$.

In view of the estimates derived above, taking the supremum in
\eqref{eq:h-hdelta} over $\abs{\z}\le \delta$, we have established
\begin{align*}
	\omega(\delta)
	& \le C\, \delta^{-d}\int_{\abs{y}\le \frac{\delta}{2}}
	\int_{\R^d} \abs{h(\x+\y)-h(\x)}\psi(\x) \,d\x\, d\y
	\\ & \quad\quad +
	C \norm{h}_{L^1(\R^d)} \Bigl(\delta \, \mathbf{1}_{0\le \delta \le 1}(\delta)
	+ \mathbf{1}_{\delta>1}(\delta)\Bigr).
\end{align*}

Multiplying this by $\delta^{-r-1}$ and
integrating yields (replacing $\y$ by $\z$)
\begin{equation}\label{eq:mod-tmp2}
	\begin{split}
		&\int_0^\infty \delta^{-r-1}\omega(\delta)\, d\delta
		\\ & \quad
		\le  C \,
		\int_0^\infty\delta^{-r-1-d}\int_{\abs{\z}\le \frac{\delta}{2}}
		\int_{\R^d} \abs{h(\x+\z)-h(\x)}\psi(\x)\,d\x\,d\z\, d\delta
		\\ & \quad\qquad
		+C \norm{h}_{L^1(\R^d)} \left(\int_0^1 \delta^{-r} \, d\delta
		+ \int_1^\infty \delta^{-r-1}\,d\delta\right) =: A+B,
	\end{split}
\end{equation}
where the integrals on the last line are bounded since $r\in (0,1)$:
\begin{align*}
	B \le C_r\, \norm{h}_{L^1(\R^d)}.
\end{align*}

Since $\frac{\abs{\z}}{\delta}\le \frac12
\Rightarrow J\left(\tfrac{\abs{\z}}{\delta}\right)>0$
and remembering $r<s$,
\begin{align*}
	&A \le  C_J\int_0^1\delta^{-r-1-d}\int_{\abs{\z}\le \frac{\delta}{2}}
	\int_{\R^d} \abs{h(\x+\z)-h(\x)}J\left(\tfrac{\abs{\z}}{\delta}\right)
	\psi(\x)\,d\x\,d\z\, d\delta
	\\ & \quad \le  C_J\int_0^1\delta^{-s}\delta^{s-r-1}\int_{\abs{\z}\le \frac{\delta}{2}}
	\int_{\R^d} \abs{h(\x+\z)-h(\x)}J_\delta(\z)\psi(\x)\,d\x\,d\z\, d\delta
	\\ & \quad \le  C_J\left(\int_0^1 \frac{1}{\delta^{1+r-s}}\, d\delta\right)
	\sup_{0<\delta\le 1}\, \delta^{-s}\int_{\R^d}\int_{\R^d}
	\abs{h(\x+\z)-h(\x)}J_\delta(\z)\,d\x\,d\z
	\\ & \quad \le  C_{J,r,s}\,
	\sup_{0<\delta\le 1}\, \delta^{-s}\int_{\R^d}\int_{\R^d}
	\abs{h(\x+\z)-h(\x)}J_\delta(\z)\, \psi(\x)\,d\x\,d\z,
\end{align*}
where  $C_{J,r,s}=C_J\, \frac{1}{s-r}$.

Consequently, from \eqref{eq:mod-tmp1}
and \eqref{eq:mod-tmp2} it follows that
for any $\delta>0$,
\begin{equation}\label{eq:mod-tmp3}
	\begin{split}
		&\sup_{\abs{z}\le \delta}\,
		\int_{\R^d} \abs{h(\x+\z)-h(\x)}\psi(\x)\, d\x
		\\ & \quad \le  C\, \delta^{r}
		\sup_{0<\delta\le1}\, \delta^{-s}\int_{\R^d}\int_{\R^d}
		\abs{h(\x+\z)-h(\x)}J_\delta(\z)\psi(\x)\,d\x\,d\z
		\\ & \quad \qquad
		+ C\, \delta^{r} \norm{h}_{L^1(\R^d)},
	\end{split}
\end{equation}
for some finite constant $C$.

Finally, observe that
\begin{align*}
	& \int_{\R^d}\int_{\R^d}
	\abs{h(\x+\z)-h(\x)}J_\delta(\z)\psi(\x)\,d\x\,d\z
	\\ & \quad =  \int_{\R^d}\int_{\R^d} \abs{h(\x+\z)-h(\x-\z)}
	J_\delta(2\z)\psi(\x-\z)\,d\x\,d\z
	\\ &\quad  =  \frac{1}{2^d}\int_{\R^d}\int_{\R^d}
	\abs{h(\x+\z)-h(\x-\z)}J_{\delta/2}(\z)\psi(\x-\z)\,d\x\,d\z
	\\ & \quad \le  \frac{1}{2^d}\int_{\R^d}\int_{\R^d}
	\abs{h(\x+\z)-h(\x-\z)}J_{\delta/2}(\z)\psi(\x)\,d\x\,d\z
	+ I_3(\delta),
\end{align*}
where $I_3(\delta)$ denotes the expression
$$
\frac{1}{2^d} \int_{\R^d}\int_{\R^d}
\abs{h(\x+\z)-h(\x-\z)}J_{\delta/2}(\z)\abs{\psi(\x)-\psi(\x-\z)}\,d\x\,d\z.
$$
As with $I_1(\delta)$, $I_3(\delta) \le
C \norm{h}_{L^1(\R^d)}  \Bigl(\delta\,
\mathbf{1}_{0\le \delta \le 1}(\delta)
+\mathbf{1}_{\delta>1}(\delta)\Bigr)$,
and as a result
\begin{equation*}
	\begin{split}
		& \sup_{0<\delta\le1}\,
		\delta^{-s}\int_{\R^d}\int_{\R^d}
		\abs{h(\x+\z)-h(\x)}J_\delta(\z)\psi(\x)\,d\x\,d\z
		\\ &\quad  \le C \sup_{0<\delta\le1}\, \delta^{-s}
		\int_{\R^d}\int_{\R^d}
		\abs{h(\x+\z)-h(\x-\z)}J_{\delta/2}(\z)\psi(\x)\,d\x\,d\z
		\\ & \quad \qquad + C \norm{h}_{L^1(\R^d)}.
	\end{split}
\end{equation*}
We can therefore replace \eqref{eq:mod-tmp3} by
\begin{equation*}
	\begin{split}
		&\sup_{\abs{z}\le \delta}\,
		\int_{\R^d} \abs{h(\x+\z)-h(\x)}\psi(\x)\, d\x
		\\ & \quad \le  C\, \delta^{r}
		\sup_{0<\delta\le1}\, \delta^{-s}\int_{\R^d}\int_{\R^d}
		\abs{h(\x+\z)-h(\x)}J_{\frac{\delta}{2}}(\z)\psi(\x)\,d\x\,d\z
		\\ & \quad \qquad
		+ C\, \delta^{r} \norm{h}_{L^1(\R^d)},
	\end{split}
\end{equation*}
for some finite constant $C$, which implies \eqref{eq:trans-to-sob}.

\medskip
\bigskip
{\bf Acknowledgements}.
The research of Gui-Qiang Chen was
supported in part by the National Science
Foundation under Grants DMS-0935967
and DMS-0807551, the UK EPSRC Science and Innovation
Award to the Oxford Centre for Nonlinear PDE (EP/E035027/1),
the NSFC under a joint project Grant 10728101, and
the Royal Society--Wolfson Research Merit Award (UK).
The research of Qian Ding was supported in part by the National Science
Foundation under Grants DMS-0807551 and the UK EPSRC Science and Innovation
Award to the Oxford Centre for Nonlinear PDE (EP/E035027/1).

\end{document}